\RequirePackage{fix-cm}
\documentclass[smallextended]{svjour3}       
\smartqed  
\usepackage{graphicx}
\usepackage{xcolor}
\usepackage{makecell}
\usepackage{graphicx,times}
\usepackage{subfigure}         
\usepackage{diagbox}         
\usepackage{subfigure}         
\usepackage{graphicx}
\usepackage{amssymb}
\usepackage{lineno}
\usepackage{algorithm}
\usepackage{setspace}
\usepackage{multirow}
\usepackage{xcolor}
\usepackage{booktabs}
\usepackage{algorithmic}
\usepackage{indentfirst}
\usepackage{amsmath}
\usepackage{mathrsfs}
\usepackage{bm}
\usepackage{framed} 
\usepackage{adjustbox}
\usepackage[graphicx]{realboxes}
\usepackage{mathptmx}      
\begin{document}

\newcommand{\diag}{{\rm diag}}
\newcommand{\argmin}{{\rm argmin}}
\newcommand{\off}{{\rm off}}
\newcommand{\sgn}{{\rm sgn}}
\newcommand{\spann}{{\rm span}}
\newcommand{\st}{{\rm s.t.}}
\newcommand{\trace}{{\rm trace}}
\renewcommand{\O}{{\cal O}}
\renewcommand{\P}{{\cal P}}
\newcommand{\R}{{\cal R}}
\newcommand{\V}{{\cal V}}
\newcommand{\sign}{{\rm sign}}

\title{Exterior Point Method for Completely Positive Factorization 
}

\author{Zhenyue Zhang         \and
        Bingjie Li 
}


\institute{Zhenyue Zhang \at
School of Mathematics Science, Zhejiang University, 
Yuquan Campus, Hangzhou 310038, China \\
Research Center for Advanced Artificial Intelligence Theory, Zhejiang Lab, Hangzhou 311121, China\\
\email{zyzhang@zju.edu.cn}           
\and
Bingjie Li \at
School of Mathematics Science, Zhejiang University, 
Yuquan Campus, Hangzhou 310038, China \\
\email{11535017@zju.edu.cn}  
}

\date{Received: date / Accepted: date}

\maketitle

\begin{abstract}
 Completely positive factorization (CPF) is a critical task with applications in many fields. This paper proposes a novel method for the CPF. Based on the idea of exterior point iteration, an optimization model is given, which aims to orthogonally transform a symmetric lower rank factor to be nonnegative. The optimization problem can be solved via a modified nonlinear conjugate gradient method iteratively. The iteration points locate on the exterior of the orthonormal manifold and the closed set whose transformed matrices are nonnegative before convergence generally. Convergence analysis is given for the local or global optimum of the objective function, together with the iteration algorithm. Some potential issues that may affect the CPF are explored numerically. The exterior point method performs much better than other algorithms, not only in the efficiency of computational cost or accuracy, but also in the ability to address the CPF in some hard cases. 


\end{abstract}

\section{Introduction}\label{intro}

Completely positive factorization (CPF) of a nonnegative, symmetric, and positive semidefinite matrix $A$ looks for a nonnegative factorization $A = BB^T$ with a nonnegative factor $B$.
If such a CPF exists, the matrix $A$ is said to be completely positive. More preciously, a {\it strict} CPF of $A$ means the factorization $A = BB^T$ with a nonnegative factor $B$ having the smallest column number $r_{cp}$ since such a factorization may be not unique if it exists. The smallest column number $r_{cp}$ is also called as the completely positive rank, or cp-rank for short, of $A$.

The CPF problem appears in many fields such as balanced incomplete block designing \cite{BH1979}, energy demand model designing \cite{GW1980}, probability estimation of DNA sequences \cite{K1994}. In matrix games, completely positive matrices are used to check if the Nash equilibrium can be found by algorithm Lemke that solves a linear complementary problem \cite{MZ1991}. In the recent decade, the CPF plays an important role in combinatorial and nonconvex quadratic optimization. For example, some binary quadratic problem or graph optimization can be relaxed as a continuous optimization problem over a set of completely positive matrices \cite{B2009,B2012,D2010}, and the iteration of completely positive matrices could be implemented in a cone where the updated completely positive matrices are given in the form of convex combination of rank-one nonnegative matrices \cite{BD2009}, or via optimizing the nonnegative factors of the CPF \cite{BJR2011}. In the latter case, an initial CPF should be given for running the factor-optimization. 
In data science, CPF or completely positive approximation can also be used to refine noisy graphs for data clustering \cite{KDP2012,KYP2015,XXS2010} or to estimate the probabilities of extreme events happen in nature or financial markets \cite{CT2019}.
Hence, the CPF is interesting in both of theory analysis and applications. 

However, the CPF problem is NP-hard in the following issues \cite{BDS2015}: detecting whether a nonnegative, symmetric, and positive semidefinite matrix is completely positive \cite{DG2012}, determining the cp-rank of a completely positive matrix, finding a nonnegative factor $B$ when the cp-rank $r_{cp}(A)$ of $A$ is known \cite{GD2018}, 
or solving the problem $\min_{B\geq 0} \|A-BB^T\|_F^2$ for completely positive approximation, where $B$ has a given number of columns \cite{VGL2016}. 
Only for a narrow part of completely positive matrices with special structures, the CPF can be obtained directly or within polynomial time. For example, diagonal dominant nonnegative symmetric matrices are completely positive \cite{K1987}. 
A bipartite graph, {\it i.e.}, $A = \left[\begin{array}{ll}
     D_1 & C^T \\
     C   & D_2
\end{array}
\right]$ with positive diagonal $D_1$ and $D_2$ and nonnegative $C$, is completely positive and has an explicit CPF \cite{BG1988}. 
For an acyclic or circular graph, if it is positive semidefinite, one can check whether it is completely positive or not within polynomial time. Furthermore, if it is completely positive, one can also get its CPF within polynomial time \cite{DD2012}. 
It was shown in \cite{KG2012} that if a completely positive matrix of rank $r$ has a principal diagonal matrix of order $r$, then its cp-rank equals to $r$ and its CPF can be obtained explicitly. 

There are some efforts on algorithms for this problem in the literature. Moody Chu proposed a direct approach in \cite{DLM2014} via successive rank-one reduction, based on the Wedderburn's formula to reduce the rank \cite{W1934}: $\hat A = A-\sigma^{-1}uu^T$ with $u = Ax$ and $\sigma = x^TAx$. Here $x$ should yield a nonnegative $u$ and a positive $\sigma$, and meanwhile, $\hat A$ is also nonnegative.\footnote{The procedure terminates if such an $x$ cannot be obtained.} A greedy approach was given in \cite{DLM2014} to determine such an $x$ via maximizing the smallest entries of $\hat A$. Kuang et al. considered the natural model $\min_{W\geq 0} \|A-WW^T\|_F^2$ and solved it by projected Newton method in \cite{KDP2012}. This problem was relaxed to the penalty form $\min_{W,H \geq 0}\|A-WH^T\|_{F}^2+\lambda \|W-H\|_F^2$ in  \cite{KYP2015} so that it can be solved iteratively via alternatively optimizing the factors $W$ and $H$, each is a nonnegative least square problem without a closed form solution. 
Writing $\|A-WW^T\|_F^2 = \|A_j-w_{\cdot j}w_{\cdot j}^T\|^2_2$ for each $j$, where $w_{\cdot j}$ is the $j$-th column of $W$ and $A_j = A-\sum_{i\neq j}w_{\cdot i}w_{\cdot i}^T$,\footnote{We will use $w_i$ for the $i$-th row of $W$.} In \cite{VGL2016}, Vandaele et al. considered the successive column-updating: $\min_{w_{\cdot j}\geq 0} \|A_j-w_{\cdot j}w_{\cdot j}^T\|_F^2$, sweeping all the columns of $W$ repeatedly until convergence. 

Different from the direct optimization on nonnegative factors, \cite{HSS2014} and \cite{GD2018} considered two approaches that orthogonally transform a full-rank factor $W$ of the symmetric factorization $A = WW^T$ to be nonnegative as $B=WQ$ directly or indirectly. In \cite{HSS2014}, the explicit error $\|B-WQ\|_F^2$ is minimized, subjected to nonnegative $B$ and orthogonal $Q$.
This problem is solved by alternatively optimizing nonnegative $B$ and orthogonal $Q$, and the iteration converges to a point satisfying its first order condition of local optimum. 
The optimization considered in \cite{GD2018} is implemented in the original space as $\min\|Q-P\|_F^2$ subjected to orthogonal $Q$ and restricted matrix $P$ with nonnegative image $WP$. This problem is also solved by alternative iteration on $Q$ and $P$. However, it is hard to solve the subproblem on $P$ because of the implicit restriction on $P$.\footnote{An approximate rule of $P$ was suggested in \cite{GD2018} to simplify the updating, but $WP$ may be no longer nonnegative.}
Local convergence of these alternative projection algorithms was given in \cite{GD2018}. However, it is not clear if the local convergence could be linear. In \cite{D2013}, Drusvyatskiy proved that for an alternative projection onto two closed sets $\cal X$ and $\cal Y$, the linear convergence occurring nearby an intersected point $z$ of $\cal X$ and $\cal Y$ if $\cal X$ intersects $\cal Y$ at $z$ transversally. That is, at $z$, the normal cone of $\cal X$ and the negative normal cone of $\cal Y$ are intersected at the origin only. For the alternative projection considered in \cite{GD2018}, this condition at an orthogonal $Q$ with nonnegative $B=WQ$ is equivalent to that for any nonnegative matrix $Y\in {\mathbb R}^{n\times r}$, $S = B^TY$ is not symmetric or its trace is not zero if $S\neq 0$. See Proposition \ref{prop:QP} given in Appendix A. This condition is not satisfied in some cases.
Below is an instance: the completely positive matrix $A$ has a factor $W$, and there is an orthogonal $Q$ with $B=WQ\geq 0$ and a $Y\geq 0$ such that $B^TY$ is symmetric and nonzero. However, it has a zero trace. 
\[A = \left[\begin{matrix}6 & 2 & 6 & 2\\
                     2 & 8 & 2 & 6\\
                     6 & 2 & 9 & 1\\
                     2 & 6 & 1 & 5\\
 \end{matrix}\right], \quad
 W =
 \left[\begin{matrix}-1 & 1  & 2 & 0  \\
                     -2 & 0  & 0 & -2 \\
-                      0 & 2  & 2 & -1 \\
                     -2 & 0  & 0 & -1 \\
 \end{matrix}\right],\quad
 Q =
 \left[\begin{matrix}0 &  0  & -1 & 0  \\
                     0 &  0  &  0 & 1 \\
                     1 &  0  &  0 & 0 \\
                     0 & -1  &  0 & 0 \\
 \end{matrix}\right],\quad
 Y=
 \left[\begin{matrix}0 & 1 & 0 & 0\\
                     1 & 0 & 0 & 0\\
                     0 & 0 & 1 & 0\\
                     0 & 0 & 0 & 1\\
 \end{matrix}\right].
\]
These algorithms mentioned above may suffer from the nonnegative or orthogonal restrictions. Because of the restrictions, an inner iteration is required to solve the involved subproblems if there are no closed form solutions, or adopt an estimated solution instead, but it may delay the convergence. 

In this paper, we consider a novel approach for the CPF without restrictions. Different from the alternative projection mentioned above, our idea is {\it exterior point iteration} that pursues an orthogonal $Q$ on the exterior of the set of orthogonal matrices and the set of $P$'s whose images $WP$ are nonnegative. It may be expected that such an exterior point iteration is more efficient since no restrictions should be obeyed.
We will give an optimization model for implementing the exterior point iteration. 
Decreasing the objective function at a point $X$ implicitly propels $X$ moving toward the set of orthogonal matrices, and meanwhile, its image $WX$ is nonnegative eventually. We will show some properties about the first order condition and the second order condition for the optimal solutions. We will also characterize the global optimum that can be used to check whether the iterative point will go to a global optimal one or not, when we solve this optimization problem iteratively. Based on this property, a restart strategy can be also used in order to avoid local optimum as much as possible. 
The classical nonlinear conjugate gradient (NCG) method requires a twice continuously differentiable function \cite{A1985,SY2006} or differentiable function with Lipschitz continuous gradient \cite{SY2006} to guarantee the convergence. We modify the NCG so that the condition of the objective function can be weakened to be continuously differentiable
as that in our model. The modified NCG performs very efficiently in our experiments.

It is not clear what is the dominant issue that determines the difficulty of CPF. In this paper, we will numerically explore the three possible issues: the cp-rank of $A$, the sparsity of the nonnegative factor $B$ in a CPF of $A$, or the approximately cp-rank deficiency of $A$. An interesting observation is that there is a special rank-sparsity boundary nearby which the CPF is much harder than others. 
We believe that these phenomena offer some important insight into the CPF problem. 

The exterior point method performs much better than other algorithms given in the literature, not only on the suitableness of completely positive matrices whenever whose cp-rank is equal to or larger than the rank, but also on the computational efficiency on the accuracy of the factorization or the computational cost for completely positive matrices in large scales. We will report the results of the numerical experiments to show the advantages and comparisons with other algorithms. 

The remaining part of this paper is organized as follows. In Section \ref{sect:epm}, we show the motivation of the exterior point method and the optimization model of the CPF. Then the first order and second order conditions are discussed. We also discussed some interesting properties of the global optimum. The NCG method and its modifications are discussed in Section \ref{sect:mncg} for solving the problem. In Section \ref{sect:explore}, we report some interesting phenomena of the CPF via a lot of numerical experiments by the modified NCG, which partially explores possible issues resulting in difficult CPF. We also show some improvements for difficult CPF by weak CPF or restart techniques in Section \ref{sect:improvement}. Comparisons with other algorithms are given in Section \ref{sect:comparison} that further show the efficiency of our exterior point method on synthetic completely positive matrices in small or larger scales and some difficult examples reported in the literature. Finally, some remarks are given in the conclusion section. 

\section{The Exterior Point Method}\label{sect:epm}

Let $A$ be a completely positive matrix of order $n$, and let $r = r(A)$ and $r_{cp} = r_{cp}(A)$ be the matrix rank and cp-rank of $A$, respectively. It is easy to estimate the matrix rank but hard for the cp-rank unless $r=1,2$.\footnote{In this special case, $r_{cp} = r$.} Theoretically, the cp-rank could be equal to or much larger than the matrix rank or the matrix order. There are two kinds of estimations on the cp-rank. One was an upper bound of $r_{cp}$ in term of $r$ given in \cite{BB2003}: $r_{cp}\leq \frac{1}{2}(r^2+r-2)$ for $r\geq 2$. The estimation is not tight when $n$ is relatively large. For instance, the upper bound cannot be touched if $A$ is of full rank with $n>5$. This claim can be concluded from the other kind of estimation in terms of matrix order given in \cite{MBB2015}:  $r_{cp}\leq \frac{1}{2}(n^2+n-8)$ if $n\geq 6$, not depending on the rank of $A$. A lower-bound given in \cite{BSU2015} as that $r_{cp}\geq \frac{1}{2}\big(n^2+n-8-(\sqrt{8n}-3)n\big)$ for $n\geq 15$. 
In the literature, it is commonly assumed that $r_{cp}=r$. It makes sense since if we arbitrarily choose $r$ nonnegative $n$-dimensional vectors with $r\leq n$, these vectors are almost always linearly independent, {\it i.e.}, the matrix $B$ of these vectors is of full rank, and hence $A = BB^T$ has rank $r$ almost always. 

In this section, we do not assume that $r_{cp} = r$. Furthermore, we do not ask the column number of $B$ equal to $r_{cp}$ in the CPF that we are going to determined. Similar with that in \cite{HSS2014} and \cite{GD2018}, we look for an orthonormal transformation $B = WQ$ from a full-rank factor $W$ of a symmetric factorization $A = WW^T$. That is, $W$ is of $r$ columns, $B$ has $r_+$ columns with a given integer $r_+ \geq r_{cp}$, and $Q$ is a row-orthonormal matrix of order $r\times r_+$. Such an orthonormal matrix exists according to the following lemma. A similar result was given in \cite{HSS2014}, here we give a much simpler proof.

\begin{lemma}\label{lma:orth T}
Let $A = WW^T$ be a full column-rank and symmetric factorization of $A$, and $r_+ \geq r_{cp}(A)$. Then $A = BB^T$ with a nonnegative $B$ of $r_+$ columns, if and only if there is a row-orthonormal $Q$ of order $r\times r_+$ such that $B = WQ$.
\end{lemma}
\proof
The sufficiency is obvious since $QQ^T=I$. For the necessity, the equality $BB^T=WW^T$ implies that $B$ and $W$ have the same rank and $\spann(W)=\spann(B)$. Therefore, $B = WQ$ for a matrix $Q\in\mathbb{R}^{r\times r_+}$. Substituting $B = WQ$ into $BB^T=WW^T$, we get that $QQ^T = I$ since $W$ is of full column rank. The proof is then completed.
$\hfill\square$
\endproof

\subsection{The optimization model for CPF}\label{subsect:OPM}

The alternative projection algorithms mentioned before look for two sequences $\{Q_k\}$ and $\{P_k\}$ in the constrained domains
\[
    \mathbb{Q}=\{Q\in{\mathbb{R}}^{r\times r_+}: QQ^T = I_r\},\quad
    \mathbb{P}=\{P\in{\mathbb R}^{r\times r_+}: WP\geq 0\},
\]
respectively, to pursue an intersect point of the feasible domains, via alternative projection. 
Different from the greedy alternative projection, we consider a new strategy of {\it exterior point iteration} for solving the problem of CPF. The idea is that we look for an iterative sequence $\{X_k\}$ that alive on the exterior of the two subdomains. Each $X_k$ is not orthogonal and $WX_k$ may be not nonnegative. Meanwhile, we hope that the sequence $\{X_k\}$ can get close to both of $\mathbb{Q}$ and $\mathbb{P}$ more and more, and eventually, the sequence $\{X_k\}$ can converge to an intersect point $Q$ of $\mathbb{Q}$ and $\mathbb{P}$ to get a nonnegative factor $B=WQ$ of $A$. 

To this end, let us consider the distances of a given matrix $X$ to $\mathbb{Q}$ and $\mathbb{P}$, respectively,
\[
    d(X,\mathbb{Q}) = \min_{Q\in\mathbb{Q}}\|X-Q\|_F,\quad
    d(X,\mathbb{P}) = \min_{P\in\mathbb{P}}\|X-P\|_F.
\]
The ideal model for the CPF is 
\begin{align}\label{prob:ideal}
    \min_{X\in {\mathbb{R}}^{r\times r_+}} \ \big\{d^2(X,\mathbb{Q})+d^2(X,\mathbb{P})\big\}.
\end{align}
Clearly, $d^2(X,\mathbb{Q})$ has a simple representation 
$d^2(X,\mathbb{Q}) = \sum_k(1-\sigma_k(X))^2$, where $\{\sigma_k(X)\}$ are the singular values of $X$. Unfortunately, there is not a closed form to represent the distance $d(X,\mathbb{P})$ because of the implicit restriction in $\mathbb{P}$. However, $d(X,\mathbb{P})$ can be represented by $\|WX-WP\|_F$ since $\|WX-WP\|_F\leq \|W\|_2\|X-P\|_F$. It has a lower bounded as
\[
    d(X,\mathbb{P})
    \geq \min_{P\in\mathbb{P}}\frac{\|WX-WP\|_F}{\|W\|_2}
    \geq \min_{B\in \mathbb{R}_+^{n\times r_+}}\frac{\|B-WX\|_F}{\|W\|_2}
    = \frac{\|(WX)_-\|_F}{\|W\|_2},
\]
where $(WX)_- = \min\big\{WX, 0\big\}$ is the negative part of $WX$. Therefore, the ideal objective function can be equivalently transformed to 
\begin{align}\label{prob:opt0}
    \min_{X\in {\mathbb{R}}^{r\times r_+}} \ 
    \Big\{\sum_k\big(1-\sigma_k(X)\big)^2 +\frac{\|(WX)_-\|_F^2}{\|W\|_2^2}\Big\}.
\end{align}

However, the above model should be slightly modified in the view of numerical computation since it costs much to evaluate all the singular values. This disadvantage can be addressed by taking into account the equality $\|XX^T-I\|_F^2 = \sum_i(1-\sigma_i^2(X))^2$. Since the global optimal solution is achieved at an $X$ with $\sigma_i = 1$, when $\sigma_i(X)\approx 1$ for all $i$, we see that $\|XX^T-I\|_F^2 = \sum_i(1-\sigma_i(X))^2(1+\sigma_i(X))^2\approx 4\sum_i(1-\sigma_i(X))^2$.
More importantly, it can be evaluated economically. Notice that $\|XX^T-I\|_F^2$ is a quartic function of $X$, while $\|(WX)_-\|_F^2$ is quadratic. We slight modify the coefficients of the two functions in (\ref{prob:opt0}) when the first one is replaced by $\|XX^T-I\|_F^2$, 
\begin{align}\label{prob:epm}
    \min_{X\in {\mathbb{R}}^{r\times r_+}} \ 
    \Big\{ \frac{1}{4}\|XX^T-I\|_F^2 +\frac{\lambda}{2}\|(WX)_-\|_F^2\Big\}.
\end{align}
where $\lambda = 2/\|W\|_2^2$. One may slightly change $\lambda$ if necessary. Numerically, it may be more robust if we rescale the rows of $W$ to have unit norm, {\it i.e.}, replace $W$ by its normalized 
\[
    \tilde W = \diag(\|w_1\|^{-1},\cdots,\|w_n\|^{-1})W
\]
in (\ref{prob:epm}), where $\{w_i\}$ are the rows of $W$. Since $\sqrt{n} = \|\tilde W\|_F\leq \sqrt{r}\|\tilde W\|_2$, $1/\|\tilde W\|_2^2\leq\frac{r}{n}$. Hence, we can simply set $\lambda\approx \frac{2r}{n}$ in this case. In our experiments, we always use this parameter set for $\lambda$ and it works very well. We can set $B = (WX)_+$, the nonnegative part of $WX$, as the required nonnegative factor.

There are some advantages to the problem (\ref{prob:epm}) and its objective function 
\[
    f(X) = \frac{1}{4}\|XX^T-I\|_F^2 +\frac{\lambda}{2}\|(WX)_-\|_F^2.
\]
At first, $X\in \mathbb{Q}\cap\mathbb{P}$ if and only if $f(X) = 0$, {\it i.e.}, it is an optimal solution of (\ref{prob:epm}). Second, because of the square form of the negative entries in the second term, $f$ is derivative. Third, if at a point $X_0$, $WX_0$ does not have zero entries, $f(X)$ is a polynomial function near $X_0$. This property implies a similar behavior as a polynomial function that benefits the analysis of convergence or global optimality. In the next subsection, we will further discuss the local or global optimality.

\subsection{Conditions for local optimum}

Because the objective function $f$ is derivable, it is not difficult to characterize its first order condition for its local optimum. One can verify that
\begin{align*}
	\frac{\partial }{\partial x_{ij}}\big\|(WX)_-\big\|_F^2 
	&= \frac{\partial }{\partial x_{ij}} 
	    \sum_{s=1}^{n}\Big(\sum_{t=1}^{r}w_{st}x_{tj}\Big)_-^2
    = 2\sum_{s=1}^{n}w_{si}\Big(\sum_{t=1}^{r}w_{st}x_{tj}\Big)_-.
\end{align*}
The gradient of $\|(WX)_-\|_F^2$ can be represented as $\nabla \|(WX)_-\|_F^2 = 2W^T(WX)_-$. We also have that $\nabla\|XX^T-I\|_F^2 = 4(XX^T-I)X$. Hence, 
\begin{align}\label{f:der}
    \nabla f(X) = (XX^T-I)X+\lambda W^T(WX)_-.
\end{align}
The first order condition for a local minimum of $f(X)$ follows immediately.

\begin{lemma}\label{lma:KKT}
The first order condition of a (local) minimum of $f$ is 
\begin{align}\label{first order}
(XX^T-I)X + \lambda W^T(WX)_- = 0.
\end{align}
\end{lemma}

Some interesting properties follows for any stationary point of $f$, {\it i.e.}, a matrix $X$ satisfying (\ref{first order}). The first one is that the term $\lambda\|(WX)_-\|_F^2$ and $f$ can be represented in terms of the singular values $\{\sigma_i(X)\}$ of $X$.

\begin{proposition}\label{prop:stationary point}
Let $X$ be a stationary point of $f$ and $\sigma_i = \sigma_i(X)$ for simplicity. Then
\begin{align}\label{f}
    \lambda\|(WX)_-\|_F^2 
    =\sum_i\sigma_i^2(1-\sigma_i^2),\quad
    f(X) 
    = \frac{1}{4}\sum_i(1-\sigma_i^4). 
\end{align}
\end{proposition}
\begin{proof}
Since $\trace(X^T(I-XX^T)X) = \trace(X^TX(I-X^TX)) = \sum_i\sigma_i^2(1-\sigma_i^2)$, 
\[
    \lambda\|(WX)_-\|_F^2 = \trace(\lambda (WX)^T(WX)_- )
    = \trace(X^T(I-XX^T)X) = \sum\sigma_i^2(1-\sigma_i^2).
\]
This is the first inequality in (\ref{f}). The second one follows it with the definition of $f$ and $\|XX^T-I\|_F^2 = \sum_i(1-\sigma_i^2)^2$ and  $(1-\sigma_i^2)^2+2\sigma_i^2(1-\sigma_i^2)=1-\sigma_i^4$.
$\hfill\square$
\end{proof}

Obviously, all $\sigma_i(X)=1$ if $X$ is globally optimal. One can also expect $\sigma_i(X)\approx 1$ even if $X$ is just a locally optimal or stationary point since $\lambda\|(WX)_-\|_F^2>0$ and 
\[
    \sum_{\sigma_i>1}\sigma_i^2(\sigma_i^2-1)
    = \sum_{\sigma_i<1}\sigma_i^2(1-\sigma_i^2)
    -\lambda\|(WX)_-\|_F^2.
\]
Practically, one can conclude form (\ref{first order}) that some
$\sigma_i(X)=1$ if $WX$ has partial nonnegative columns because of the symmetry of $\lambda (WX)^T(WX)_- = X^T(I-XX^T)X$. If $(WX)_+$ is the nonnegative part of $WX$, the symmetry gives
\begin{align}\label{WX}
    (WX)_+^T(WX)_- = (WX)_-^T(WX)_+.
\end{align}

\begin{proposition}\label{prop:X_12}
For a stationary point $X$ of $f$, if $WX$ has a nonnegative submatrix $WX_1$ and the remaining $WX_2$ does not, then there are at least $r(X_1)$ singular values $\sigma_i(X)=1$, and
\[
    (WX_1)^T(WX_2)_-= 0, \quad (WX_1)^T(WX_2)\geq 0.
\]
\end{proposition}
\begin{proof}
By (\ref{first order}), $(XX^T-I_r)X_1=0$. Hence, $XX^T$ has at least $r_1=r(X_1)$ eigenvalues equal to one. It follows that $X$ has at least $r_1$ singular values equal to one. 
By (\ref{WX}),
\[
    (WX_1)_+^T(WX_2)_-=(WX_1)_-^T(WX_2)_+=0.
\]
It follows that $(WX_1)^T(WX_2)=(WX_1)_+^T(WX_2)=(WX_1)_+^T(WX_2)_+\geq0$.
$\hfill\square$
\end{proof}

The orthogonality between $WX_1$ and $(WX_2)_-$ implies that the rows of $WX_2$ corresponding to the nonzero rows of $WX_1$ must be nonnegative. Hence, if $WX$ has a positive column, $WX$ itself must be nonnegative. If $WX_2$ exists, $WX_1$ must have zero rows, and we can partition $WX$ as $WX=\left[\begin{array}{cc}
     W_1X_1 & W_1X_2 \\
     0 & W_2X_2
\end{array}\right]$ within permutation, where both $W_1X_1$ and $W_1X_2$ are nonnegative and all the rows of $W_1X_1$ are nonzero.

A second order condition can guarantee a stationary point to be local optimal. For our problem, it is not difficult to give the second order condition. 

\begin{theorem}\label{thm:second order}
Assume that $X$ is a stationary point of $f$, and $t_j$ is the indicator of negative components of the column $x_j$ of $X$.
If the symmetric matrix $S = (S_{ij})$, partitioned with blocks 
\begin{align}\label{S}
    S_{ij} = \left\{
    \begin{array}{ll}
         \|x_i\|^2I+x_ix_i^T+W^T\diag(t_i)W+(XX^T-I),& i=j; \\
         (x_i^Tx_j)I+x_jx_i^T,& i\neq j,
    \end{array}
    \right.
\end{align}
is positive definite, or positive semidefinite with a null space spanned by vectors linked by all columns of $\Delta$ satisfying $\Delta X^T+X\Delta^T=0$ and the nonzero eigenvalues of $S$ are not smaller than $\|X\|_2^2$, then $X$ is a local minimizer of $f$.
\end{theorem}

The proof is given in Appendix B. Because $f$ is not convex, local optimum happens when we solve the problem (\ref{prob:epm}) iteratively. 
Hence, it is important to check whether an iterative sequence goes to a global optimum or not. To address this issue, we will further discuss sufficient conditions for the global optimum of $f$, focusing on descending sequences $\{f(X_k)\}$ that may be generated via some iteration algorithms, as the modified NCG that we will adopt to solve the problem (\ref{prob:epm}).

\subsection{Global optimum}

We fist give some equivalent conditions for the global optimum of $f$, based on Proposition \ref{prop:stationary point} and  Proposition \ref{prop:X_12}.

\begin{proposition}\label{prop:global opt}
The following statements are equivalent for a stationary point $X$ of $f$.

(a) $X\in\mathbb P$ and $X$ is of full row rank;

(b) $X\in\mathbb Q$, {\it i.e.}, $XX^T=I_r$;

(c) $X\in\mathbb Q\cap\mathbb P$, {\it i.e.}, $X$ is a global minimizer of $f$.
\end{proposition}
\begin{proof}
The proof is simple. In fact, if $X\in\mathbb P$, {\it i.e.}, $X_1 = X$ in Proposition \ref{prop:X_12}, then all the singular values of $X$ are equal to one since $r(X) = r$, which implies that $XX^T=I_r$, {\it i.e.}, $X\in\mathbb Q$.
Conversely, if $X\in\mathbb Q$, we have $(WX)_-=0$, {\it i.e.}, $X\in\mathbb P$, from Proposition \ref{prop:stationary point} directly.
$\hfill\square$
\end{proof}

The full-rank condition on a stationary point $X$ is satisfied generally. Practically, if $f(X)< 1/4$, $X$ must be full rank by (\ref{f}). If an algorithm generates the iterative sequence $\{X_k\}$ with descending $\{f(X_k)\}$, starting with an row-orthonormal matrix, it is highly possible to have $f(X)< 1/4$ for any accumulative point $X$ of $\{X_k\}$. Hence, $X\in\mathbb P$ is equivalent to $X\in\mathbb Q$ for an accumulative point $X$ of the sequence generally. 

Proposition \ref{prop:global opt} also show that any stationary point $X$ of $f$ always locates the exterior of both $\mathbb{Q}$ and $\mathbb{P}$ except it is a solution of (\ref{prob:epm}). In Section \ref{sect:mncg}, we will give an iterative algorithm that yields a sequence of iterative points on the exterior of $\mathbb{Q}\cup\mathbb{P}$ unless it converges globally. This is why we call it as an exterior point method, quite different from the alternative methods discussed in the previous subsection, where iterative points alternatively drop into one of the subdomains.

Let us consider those stationary points of $f$ on the exterior of $\mathbb Q\cup\mathbb P$. Each $X$ of them has an indicator matrix $\sgn(|(WX)_-|)$. However, the number of different ones $\{T_\ell\}$ is limit.\footnote{Different stationary points may share a common indicator matrix.} Let $\{X^{(\ell)}\}$ be the stationary points with $\sgn(|(WX^{(\ell)})_-|)=T_\ell$, and let
\[
    f_\ell(X) = \frac{1}{4}\|XX^T-I\|_F^2 +\frac{\lambda}{2}\|(WX)\odot T_\ell\|_F^2.
\]
Each $f_\ell$ has a finite number of different critical values since it is a polynomial \cite{SH2016}. 
As shown in the proof of Theorem \ref{thm:second order} given in Appendix B, each critical value of $f$ is also a critical value of one of $f_{\ell}$'s. Hence, $f$ also has a finite number of critical values.  Let $f_{\min}$ be the smallest one of these nonzero critical values of $f$. The following theorem shows that a descending sequence $\{f(X_k)\}$ yields the optimum of $f$ if there is an $f(X_k)<f_{\min}$.

\begin{theorem}\label{thm:f_min}
Let $\{f(X_k)\}$ be a descending sequence and $\lim_k\nabla f(X_k)=0$. If there is an $f(X_k)<f_{\min}$, then $\lim_k f(X_k) = 0$.
\end{theorem}
\begin{proof}
Obviously, $\{f(X_k)\}$ converges and $\{X_k\}$ is bounded. Let $X_{\infty}$ be any accumulative point of $\{X_k\}$ and subsequence $X_{k_j}\to X_{\infty}$. Then $\nabla f(X_{\infty})=\lim \nabla f(X_{k_j}) =0$. That is, $f(X_{\infty})$ is a critical value of $f$. Since $f(X_{\infty})<f(X_k)<f_{\min}$, we conclude that $f(X_{\infty})=0$. 
$\hfill\square$
\end{proof}

Theorem \ref{thm:f_min} shows that any accumulative point of $\{X_k\}$ is a global minimizer of $f$ if there is an $f(X_k)<f_{\min}$.  It is theoretically meaningful, but the condition is not detectable since $f_{\min}$ is unknown. The following theorem gives sufficient conditions for the global optimum.

\begin{theorem}
Let $\{f(X_k)\}$ be a descending sequence such that $\lim_k\nabla f(X_k)=0$. If $\varlimsup_k\|X_k\|_2 \leq 1$, $f(X_k)<\frac{\alpha^2}{4}$ for a constant $\alpha\in(0,1)$, and for sufficiently large $k$,
\begin{align}\label{cond:WX_k}
    \lambda\|(WX_k)_-\|_F^2
    \leq \alpha\|X_kX_k^T-I\|_F - \|X_kX_k^T-I\|_F^2,
\end{align}
then any accumulative point $X_\infty$ of $\{X_k\}$ is a globally optimal solution.
\end{theorem}
\begin{proof}
Let $X_{\infty}$ be any accumulative point of $\{X_k\}$ and $X_{k_j}\to X_{\infty}$ as in the proof of Theorem \ref{thm:f_min}. 
The condition $\varlimsup_k\|X_k\|_2\leq 1$ implies that $\sigma_i\leq \|X_{\infty}\|_2\leq 1$ for all singular values $\{\sigma_i\}$ of $X_{\infty}$. Below we show that all $\sigma_i=1$, {\it i.e.}, $f(X_{\infty})=0$.

Assume that there is at least an $\sigma_i<1$. We have that $0<f(X_{\infty})
<\frac{\alpha^2}{4}$.  We will give a contradiction, based on the observation that there is a constant $\beta>0$ such that
\begin{align}\label{cond:beta}
    f(X_\infty)\leq\frac{\alpha^2(2\beta+1)}{4(\beta+1)^2},\quad 
    \lambda\|(WX_\infty)_-\|_F^2\leq \beta \|X_\infty X_\infty^T-I\|_F^2.
\end{align}

To show (\ref{cond:beta}), let $a = \|X_\infty X_\infty^T-I\|_F^2$ and $b = \lambda\|(WX_\infty)_-\|_F^2$. Both $a$ and $b$ are positive by Proposition \ref{prop:global opt} since $f(X_\infty)>0$. 
It is easy to verify that the first one in (\ref{cond:beta}) is equivalent to $\beta\leq\frac{c}{1-c}$, where $c = \sqrt{1-4 f(X_\infty)/\alpha^2}\in (0,1)$. The second one is equivalent to $\frac{b}{a}\leq\beta$. Thus, the existence of a positive $\beta$ is equivalent to $\frac{b}{a} \leq \frac{c}{1-c}$, or equivalently, $\frac{b}{a+b} \leq c$.
Recalling that the condition (\ref{cond:WX_k}) gives $a+b\leq \alpha\sqrt{a}$. Hence, 
\[
    1-c^2 = \frac{4f(X_\infty)}{\alpha^2}=\frac{a+2b}{\alpha^2} 
    \leq \frac{a(a+2b)}{(a+b)^2}
    = 1-\frac{b^2}{(a+b)^2}.
\]
That is, the inequality $\frac{b}{a+b} \leq c$ holds. Therefore, there is a positive $\beta$ satisfying (\ref{cond:beta}). 

By Proposition \ref{prop:stationary point} with $X = X_{\infty}$, the inequalities in (\ref{cond:beta}) become
\begin{align}\label{cond:beta_}
    \sum_{\sigma_i<1}(1-\sigma_i^4)\leq\frac{\alpha^2(2\beta+1)}{(\beta+1)^2},\quad
    \sum_{\sigma_i<1}\sigma_i^2(1-\sigma_i^2) \leq \beta \sum_{\sigma_i<1}(1-\sigma_i^2)^2.
\end{align}
However, the left inequality above implies that $1-\sigma_i^4<\frac{2\beta+1}{(\beta+1)^2}$ since $\alpha^2<1$, {\it i.e.}, $\sigma_i^2(\beta+1)>\beta$ for each $\sigma_i<1$. Thus, 
$\sigma_i^2>\beta(1-\sigma_i^2)$ and $\sigma_i^2(1-\sigma_i^2) > \beta (1-\sigma_i^2)^2$, which yields
a contradiction to the right inequality in (\ref{cond:beta_}), 
completing the proof.
$\hfill\square$
\end{proof}

We will show that the modified NCG given in the next section can yield a sequence $\{X_k\}$ that guarantees the descent of $\{f(X_k)\}$ and $\lim_k\nabla f(X_k)=0$. Because of the descent, it is easy to have $f(X_k)<1/4$. The inequality (\ref{cond:WX_k}) is always satisfied except a few of early $X_k$'s. Generally, $\|X_k\|_2>1$ for $k>1$, and $\|X_k\|_2\to \rho\geq 1$ in our experiments, whenever $\{X_k\}$ converges locally or globally. The difference is that $\rho>1$ for a local optimum and $\rho=1$ for the global optimum. The following theorem further characterizes an explicit and simple relation between the objective function $f$ and the norm of its derivative $\nabla f$ nearby a stationary point of $f$. This relation will be used in our algorithm to check whether an iterative sequence tends to the global optimum or not. 

\begin{theorem}\label{thm:fdf}
If $X$ is a stationary point of $f$, then for sufficiently small $\Delta = \tilde X-X$,
\begin{align}\label{f:df}
    f(\tilde X) = f(X)+\frac{1}{2}\langle \nabla f(\tilde X),\Delta \rangle 
    -\frac{1}{2}\langle \Delta\Delta^TX,\Delta \rangle-\frac{1}{4}\|\Delta\Delta^T\|_F^2.
\end{align}
Furthermore, if $X$ is globally optimal, then
$f(\tilde X)< \|\nabla f(\tilde X)\|_F\|\Delta\|_F$.
\end{theorem}
\begin{proof}
Consider an arbitrary $\tilde X$ in the open neighborhood of $X$, 
\[
    {\cal N}(X) = \big\{\tilde X:\ 
        (W\tilde X)_{ij}(WX)_{ij}>0\ \mbox{ if }\ (WX)_{ij}\neq 0\big\}.
\]
We will show that
\begin{align}
    \|\tilde X\tilde X^T\!\!-I\|_F^2 
    = &\ \big\|XX^T\!\!-I\big\|_F^2 
    +2\langle (\tilde X\tilde X^T\!\!-I)\tilde X, \Delta\rangle 
       -2\langle \lambda (WX)_-,W\Delta \rangle \nonumber\\
    &\ -2\langle X, \Delta\Delta^T\Delta\rangle
       -\|\Delta\Delta^T\|_F^2; \label{tilde_X}\\
    \|(W\tilde X)_-\|_F^2
    = &\ \|(WX)_-\|_F^2+ \langle  (W\tilde X)_-, W\Delta\rangle
    + \langle  (WX)_- , W\Delta\rangle. \label{W tilde_X}
\end{align}
Substituting these equalities into 
$f(\tilde X) = \frac{1}{4}\|\tilde X\tilde X^T-I\|_F^2
+\frac{\lambda}{2}\|(W\tilde X)_-\|_F^2$, and combining  
$\langle\nabla f(\tilde X),\Delta \rangle 
= \langle(\tilde X\tilde X^T-I)\tilde X, \Delta\rangle
+\langle\lambda (W\tilde X)_-,W\Delta \rangle$, we can get (\ref{f:df}) immediately. 

We prove (\ref{tilde_X}) below, based on $\tilde X\tilde X^T -I = XX^T-I +X\Delta^T+\Delta X^T+\Delta\Delta^T$. On one hand, from the equality, we have
\begin{align*}
	 \|\tilde X\tilde X^T-I\|_F^2 
	= &\ \big\|XX^T-I\big\|_F^2 
	    + \big\|X\Delta^T+\Delta X^T + \Delta\Delta^T\big\|_F^2\\
    &  + 2\big\langle XX^T-I,\ X\Delta^T+\Delta X^T+ \Delta\Delta^T\big\rangle.
\end{align*}
On the other hand, we rewrite 
$2\langle (\tilde X\tilde X^T-I)\tilde X, \Delta \rangle 
= \langle \tilde X\tilde X^T-I, \tilde X\Delta^T +\Delta\tilde X^T\rangle$ as
\begin{align*}
    2\big\langle (\tilde X\tilde X^T-I)\tilde X, \Delta \big\rangle
    =&\ \big\langle XX^T-I +X\Delta^T+\Delta X^T+\Delta\Delta^T,\ 
        X\Delta^T+\Delta X^T+2\Delta\Delta^T\rangle\\
    =&\ \big\langle XX^T-I,\ X\Delta^T+\Delta X^T+\Delta\Delta^T \rangle
         +\big\|X\Delta^T\!\!+\Delta X^T\!\!+\Delta\Delta^T\big\|_F^2\\
     &\quad +\langle XX^T-I +X\Delta^T+\Delta X^T+\Delta\Delta^T,\ 
        \Delta\Delta^T \rangle\\
    =&\ 2\big\langle XX^T-I,\ X\Delta^T+\Delta X^T+\Delta\Delta^T \rangle
         +\big\|X\Delta^T\!\!+\Delta X^T\!\!+\Delta\Delta^T\big\|_F^2\\
     &\quad -2\langle XX^T-I-\Delta\Delta^T,\ \Delta X^T\rangle
     + \|\Delta\Delta^T\|_F^2.
\end{align*}
Since $(XX^T-I)X=-\lambda W^T(WX)_-$ by $\nabla f(X) = 0$, 
we also have that 
\[
    \langle XX^T-I,\Delta X^T\rangle = -\langle \lambda (WX)_-,W\Delta \rangle.
\]
Hence, combining these equalities, we get (\ref{tilde_X}). 

To show (\ref{W tilde_X}), let $T_-$ and $T_0$ be the indicator matrices of the negative or zero entries of $WX$, respectively. Since $\tilde X\in{\cal N}(X)$, 
\begin{align*}
    (W\tilde X)_- = (W\tilde X)\odot (T_-+T_0)
    = (WX+W\Delta)\odot (T_-+T_0)
    = (WX)_- + H_\Delta
\end{align*}
where $H_\Delta = W\Delta\odot( T_- + T_0)$. Combining it with $\langle (WX)_-,H_\Delta\rangle=\langle (WX)_-,W\Delta\rangle$, we get
\begin{align*}
    &\|(W\tilde X)_-\|_F^2 
    =  \|(WX)_-\|_F^2 + 2\langle (WX)_-,W\Delta\rangle + \|H_\Delta\|_F^2,\\
    &\langle (W\tilde X)_-, W\Delta\rangle
    =  \langle  (WX)_- , W\Delta\rangle 
        +\langle  H_\Delta, W\Delta\rangle
    = \langle  (WX)_- , W\Delta\rangle +\|H_\Delta\|_F^2.
\end{align*}
Hence (\ref{W tilde_X}) is also true, completing the proof of (\ref{f:df}). 

For a globally optimal $X$, $f(X) = 0$ and $T_- =0$. We estimate $f(\tilde X)$ in the two subsets according to the sign of function $g(\Delta) = 2\langle\Delta^TX,\Delta^T\Delta \rangle + \|\Delta\Delta^T\|_F^2$,
\[
    {\cal N}_+ = \{\tilde X\in {\cal N}(X): g(\tilde X-X)\geq 0\}; \quad
    {\cal N}_- = \{\tilde X\in {\cal N}(X): g(\tilde X-X)<0\}.
\]
For $\tilde X \in{\cal N}_+$, we have 
$
    f(\tilde X)\leq\frac{1}{2}\langle\nabla f(\tilde X),\Delta\rangle\leq \frac{1}{2}\|\nabla f(\tilde X)\|_F\|\Delta\|_F
$
directly. 

Consider $\ell(\Delta) = X\Delta^TX+\Delta X^TX +\lambda W^TH_\Delta$ for $\tilde X = X+\Delta\in{\cal N}_-$, and rewrite 
\[
    \nabla f(\tilde X)=\ell(\Delta)+(X\Delta^T+\Delta X^T)\Delta +\Delta\Delta^T(X+\Delta)
\]
We should have $\ell(\Delta)\neq 0$ for any $\tilde X = X+\Delta\in{\cal N}_-$ different from $X$. Otherwise, $\ell(\Delta) = 0$ for a nonzero $\Delta$ such that $\tilde X = X+\Delta\in{\cal N}_-$. It follows that
\begin{align*}
    0 = \langle \ell(\Delta),\Delta\rangle 
    &\ =\langle X\Delta^TX 
        + \Delta X^TX +\lambda W^TH_\Delta, \ \Delta\rangle \\
    &\ = \langle \Delta^TX,X^T\Delta \rangle +
        \|\Delta X^T\|_F^2+\lambda \|H_\Delta\|_F^2.
\end{align*}
Since $\lambda W^TH_\Delta = -(X\Delta^T+\Delta X^T)X$ by $\ell(\Delta) = 0$, 
\[
    \|\lambda W^TH_\Delta\|_F^2 = \|X\Delta^T+\Delta X^T\|_F^2
    = 2\|\Delta X^T\|_F^2+2\langle \Delta^TX,X^T\Delta \rangle 
    = -2\lambda \|H_\Delta\|_F^2,
\]
which implies that $\Delta^TX+X^T\Delta = 0$. Thus, by $g(\Delta)<0$, we get a contradiction
\[
    0\leq \|\Delta\Delta^T\|_F^2 <
    -2\langle\Delta^TX,\Delta^T\Delta \rangle 
    = -\langle \Delta^TX+X^T\Delta,\Delta^T\Delta \rangle = 0.
\]
Therefore, $\ell(\Delta)$ is non-singular. Since it is piece-wisely linear, there is a positive constant $c$ such that $\|\ell(\Delta)\|\geq c\|\Delta\|_F$ for $\tilde X\in{\cal N}_-$. Thus, for $\|\Delta\|_F\leq\min\{1,c/5\}$, $\|\ell(\Delta)\|\geq5\|\Delta\|_F^2$, and 
by definition, $\|\nabla f(\tilde X)\|_F\geq c\|\Delta\|_F- 4\|\Delta\|_F^2\geq \|\Delta\|_F^2$. Hence,
\[
    f(\tilde X)\leq \frac{1}{2}(\|\nabla f(\tilde X)\|_F +\|\Delta\|_F^2)\|\Delta\|_F\leq \|\nabla f(\tilde X)\|_F\|\Delta\|_F.
\]
The theorem is then proven. $\hfill\square$
\end{proof}

Concluding from Theorem \ref{thm:fdf}, we get that
\begin{align}\label{dff}
    \lim_{\tilde X\to X}\frac{\|\nabla f(\tilde X)\|_F}{f(\tilde X)}
    = \left\{\begin{array}{cl}
         0,       & \ \mbox{\rm if $X$ is not globally optimal}; \\
         +\infty, & \ \mbox{\rm if $X$ is globally optimal}. 
      \end{array}\right.
\end{align}
In the next section, we will consider a nonlinear conjugated gradient method for solving the minimization problem (\ref{prob:epm}). The sufficient conditions mentioned above will be taken into account as much as possible.

\section{Modified constricted conjugate gradient algorithm}\label{sect:mncg}

The NCG method \cite{F1964} is commonly used for solving nonlinear smooth optimization problems. Generally, its convergence requires the objective function to be twice continuously differentiable \cite{A1985,SY2006}, or be differentiable and its gradient is Lipschitz continuous \cite{SY2006} if the conjugated gradient direction is suitably updated. Unfortunately, in our case, the function $f$ is continuously differentiable only.

In this section we briefly describe the NCG and conditions of its convergence at first. To guarantee the convergence when the NCG is applied to the exterior point model (\ref{prob:epm}), we propose two modifications for the NCG. One is a new approach for updating the conjugate gradient in the NCG, and the other one is a simpler rule for setting an inexact line search for updating the iteration point. These modifications can not only guarantee the convergence for continuously differentiable functions without other conditions, but also improve the efficiency of NCG. Hence, the modified NCG works on our exterior point problem (\ref{prob:epm}).

\subsection{The NCG}

The NCG provides an iterative scheme for minimizing a nonlinear smooth function $\phi(x)$ via the two classical steps, starting at $d_0  = -g_0$, with $g_0 = \nabla \phi(x_0)$ on an initial point $x_0$: 
Modify the current point $x_k$ to 
\begin{align}\label{def:x}
    x_{k+1} = x_k + \alpha_k d_k
\end{align}
along the direction $d_k$ with a suitable step length $\alpha_k$. Then, update the conjugate direction $d_k$ to 
\begin{align}\label{def:d}
    d_{k+1} = -g_{k+1} +\beta_k d_k
\end{align}
with the gradient $g_{k+1} = \nabla \phi(x_{k+1})$ at the updated point $x_{k+1}$ and a suitable value $\beta_k$.

The weak convergence $\varliminf g_{k} =  0$ or the strong convergence of $x_k$ to the minimizer of $\phi$ depends on the smoothness of $\phi$ and the choices of $\{\alpha_k\}$ and $\{\beta_k\}$. The ideal $\alpha_k$ is the minimizer of $\phi(x_k + \alpha d_k)$ with respect to $\alpha$, which is called as exact line search. Inexact search is commonly suggested but $\alpha_k$ should satisfy the weak Wolfe conditions \cite{NJWS1999}
\begin{align}
    &\phi(x_k+\alpha d_k) 
    \leq \phi(x_k)+\rho\alpha \langle g_k,d_k\rangle \label{wolfe1}\\ 
    &\langle\nabla \phi(x_k+\alpha d_k),d_k\rangle 
    \geq \sigma \langle g_k,d_k \rangle \label{wolfe2}
\end{align}
with two positive parameters $\rho<\sigma<1$, or the strong Wolfe conditions (\ref{wolfe1}) and 
\begin{align}\label{wolfe3}
    |\langle\nabla \phi(x_k+\alpha d_k),d_k\rangle| 
    \leq -\sigma \langle g_k,d_k \rangle.
\end{align}
Generally, the inexact line search $\alpha_k$ can be obtained via bisection \cite{MS1982} or interpolation \cite{SY2006}, or combination of the two approaches \cite{F2013}.
About the step choice for $\beta_k$, there are three commonly used approaches in the literature \cite{F1964,P1969,Y2009}:
\begin{align}
    \beta_k^{\rm FR} &= \frac{\|g_{k+1}\|_2^2}{\|g_k\|_2^2},\label{FR}\\ 
    \beta_k^{\rm PRP} &= \frac{\langle g_{k+1},g_{k+1}-g_k\rangle}{\|g_k\|_2^2},\label{PRP}\\
    \beta_k^{\rm MPRP} &= \beta_k^{\rm PRP} - \min\big\{\beta_k^{\rm PRP},\ \frac{\nu\|g_{k+1}-g_k\|^2}{\|g_k\|^4}
    \langle g_{k+1},d_k\rangle\big\},\quad \nu>\frac{1}{4}. \label{MPRP}
\end{align}

Sun and Yuan have shown in \cite{SY2006} that the strong convergence of the NCG with the exact line search and $\{\beta_k^{\rm PRP}\}$ is true if $\phi$ is twice continuously differentiable and uniformly convex, and the level set $\{x: \phi(x)\leq \phi(x_0)\}$ is bounded. However, the strong convergence is not guaranteed if the exact search is relaxed to the inexact one, even if it satisfies the strong Wolfe conditions. The weak convergence is guaranteed for the NCG with $\{\beta_k^{\rm FR}\}$ and inexact $\{\alpha_k\}$ satisfying the strong Wolfe conditions with $0<\rho<\sigma<\frac{1}{2}$, under the same conditions on $\phi$ without uniformly convexity, or the gradient of $\phi$ is Lipschitz continuous, a stricter condition than the continuously differentiable $\phi$. For the modified PRP, MPRP, if the weak Wolfe conditions are satisfied by $\{\alpha_k\}$ and $ \{\alpha_k\}$ has a positive lower bounded, the weak convergence can be slightly improved to $g_k\to0$ under the same assumptions on $\phi$ as for FR \cite{Y2009}. 
It was reported in the literature that the PRP is more efficient than FR in applications although its stronger conditions may not be satisfied. MPRP performs better than PRP generally if the parameter $\nu$ is suitably set. 

\subsection{Modifications for the NCG}\label{subsect:mncg}

{\bf Modified step $\beta_k$.} 
Practically, the MPRP adopts a restart strategy: reset $\beta_k=0$, {\it i.e.}, $d_{k+1}=-g_{k+1}$, when $\langle g_{k+1},g_{k+1}-g_k -\frac{\nu\|g_{k+1}-g_k\|^2}{\|g_k\|^2}d_k\rangle\leq 0$. It guarantees 
\[
    \langle d_{k+1},g_{k+1} \rangle \leq -(1- \frac{1}{4\nu}) \|g_{k+1}\|^2_2,
\]
a sufficient condition for the descent of $\phi(x_{k+1})$ by (\ref{wolfe1}). However, this descent condition is not sufficient for the convergence of $\{g_k\}$. An additional condition about the positive lower-bound of $\{\alpha_k\}$ is required \cite{Y2009}. 

To guarantee the same convergence as MPRP for continuously differentiable $\phi(x)$ without the lower-bound condition on $\{\alpha_k\}$, we further modify MPRP as that\footnote{We always assume that both $g_k$ and $d_k$ are always nonzero. NCG terminates if $g_k=0$ or restarts if $d_k=0$.}
\begin{align}\label{MMPRP}
    \beta_k = 
    \min\Big\{\frac{\langle g_{k+1},g_{k+1}-g_k
        -\frac{\nu\|g_{k+1}-g_k\|^2}{\|g_k\|^2}d_k\rangle_{+}}{\|g_k\|_2^2}, \frac{\kappa\|g_{k+1}\|_2}{\|d_k\|_2}\Big\}.
\end{align}
where $\nu>\frac{1}{4}$ as in (\ref{MPRP}) and $\kappa>0$.
The modification can guarantee a stronger sufficient descent condition
$\langle d_{k+1},g_{k+1} \rangle \leq -\mu \|d_{k+1}\|_2\|g_{k+1}\|_2$.
We will show it in the next subsection.
The stronger sufficient descent condition with $\mu\in(0,1)$ was given in \cite{SY2006} for the convergence $g_k\to 0$ if the weak Wolfe conditions are safisfied and the gradient of $\phi$ is uniformly continuous in the level set $\{x:\phi(x)\leq \phi(x_0)\}$. We will further show that the uniformly continuous condition is also not necessary for the convergence.

{\bf Simple approach for inexact line search.} 
For continuously differentiable $\phi(x)$, the interpolation method does not guarantee the capture of required weak line search $\alpha_k$ since it asks for a thrice times continuously differentiable and unimodal $\phi(x)$ \cite{SY2006}. One can get $\alpha_k$ by the combination method \cite{F2013} that is more efficient than the bisection approach \cite{MS1982}. In \cite{F2013}, the bisection is combined with the interpolation in a bit complicated way for interval shrinking. Here we give a simpler and more efficient approach for determining $\alpha_k$ satisfying the weak Wolfe conditions.

Theoretically, at a current point $x=x_k$ with the conjugate direction $d=d_k$, the required inexact line search $\alpha=\alpha_k$ satisfying the weak Wolfe conditions (\ref{wolfe1}-\ref{wolfe2}) can be chosen as
\begin{align}\label{alpha}
    \alpha^* = \sup\big\{\hat\alpha: \mbox{the Wolfe condition (\ref{wolfe1}) holds over $(0,\hat\alpha)$ }\big\}.
\end{align}
It exists, is positive, and satisfies (\ref{wolfe1}-\ref{wolfe2}). To verify this claim, let's consider the function
\[
    g(\alpha) = \phi(x)+\rho\alpha\langle\nabla \phi(x),d\rangle-\phi(x+\alpha d).
\]
Clearly, (\ref{wolfe1}) is equivalent to $g(\alpha)\geq 0$, and meanwhile, (\ref{wolfe2}) holds if $g'(\alpha)\leq0$. By the definition and the continuousness of $\phi$, (\ref{wolfe1}) is true for $0<\alpha\leq\alpha^*$. The supremum in (\ref{alpha}) implies that $g(\alpha^*)=0$ and $g'(\alpha^*)\leq 0$. Hence, (\ref{wolfe2}) is also satisfied for $\alpha=\alpha^*$. Practically, there is a relative large sub-interval of $(0,\alpha^*]$ in which both (\ref{wolfe1}) and (\ref{wolfe2}) are true. For instance, if $\hat\alpha\in(0,\alpha^*]$ is the largest point such that $g(\alpha)$ is a local maximum, then $g'(\alpha)\leq 0$ in $[\hat\alpha,\alpha^*]$. Therefore, (\ref{wolfe1}-\ref{wolfe2}) hold for $\alpha\in[\hat\alpha,\alpha^*]$. 

An ideal choice of $\alpha$ is the minimizer $\alpha_{\min}$ of $\phi(x+\alpha d)$ over $(0, \alpha^*]$ since it decreases $\phi$ as small as possible, while both (\ref{wolfe1}) and (\ref{wolfe2}) are still satisfied.
In this subsection, we give a simple rule for pursuing $\alpha_{\min}$ via a quadratic interpolation to $\phi(x+\alpha d)$, assuming $\phi$ is continuously differentiable. It generates a nested and shrunk interval sequence containing the required $\alpha$. The pursuing terminates as soon as a point satisfying (\ref{wolfe1}-\ref{wolfe2}) is found.

Initially, we set $\alpha_0'=0$ that satisfies (\ref{wolfe1}) but (\ref{wolfe2}), and choose a relatively large $\alpha_0''>0$ that does not satisfy (\ref{wolfe1}). A simple choice of $\alpha_0''$ will be given later. Starting with $[\alpha_0', \alpha_0'']$, we generate a sequence of intervals $[\alpha_0', \alpha_0'']$ iteratively such that each $\alpha_\ell'$ satisfies (\ref{wolfe1}) but  $\alpha_\ell''$ does not, and meanwhile, $\alpha_\ell'$ doesn't satisfy (\ref{wolfe2}). That is, for $x_{\ell}' = x+\alpha_\ell'd$ and $x_{\ell}'' = x+\alpha_\ell''d$
\begin{align}\label{cond_end}
    \phi(x_{\ell}')\leq \phi(x)+\rho\alpha_\ell'\langle g,d\rangle,\
    \phi(x_{\ell}'')> \phi(x)+\rho\alpha_\ell''\langle g,d\rangle,\
    \langle\nabla \phi(x_{\ell}'),d\rangle<\sigma\langle g,d\rangle,
\end{align}
where $g = \nabla\phi(x)$. The third inequality above implies that $\langle\nabla \phi(x_{\ell}'),d\rangle<0$.  
Furthermore, by the first two inequalities in (\ref{cond_end}), we have that
\begin{align}\label{cond_alpha''}
    \phi(x_{\ell}'')> \phi(x_{\ell}')+\rho(\alpha_\ell''-\alpha_\ell')\langle g,d\rangle
    >\phi(x_{\ell}')+\frac{\rho}{\sigma}(\alpha_\ell''-\alpha_\ell')\langle\nabla \phi(x_{\ell}'),d\rangle.
\end{align}
In the current interval, we consider a quadratic function $q(\alpha)$ with interpolation conditions 
\[
    q(\alpha_\ell') = \phi(x_{\ell}'), \quad 
    q'(\alpha_\ell') = \langle\nabla \phi(x_{\ell}'),d\rangle, \quad 
    q(\alpha_\ell'') = \phi(x_{\ell}''),
\]
It can be represented as 
\begin{align*}
    q(\alpha) 
    = &\ \phi(x_{\ell}')+(\alpha-\alpha_\ell') \langle\nabla\phi(x_{\ell}'),d\rangle \\
    & +\big(\phi(x_{\ell}'')-\phi(x_{\ell}')
        -(\alpha_\ell''-\alpha_\ell')\langle\nabla \phi(x_{\ell}'),d\rangle\big)
    \frac{(\alpha-\alpha_\ell')^2}{(\alpha_\ell''-\alpha_\ell')^2}
\end{align*}
with the minimizer $c_\ell = \arg\min_{\alpha}q(\alpha)$ given by
\begin{align}\label{alpha_k}
    c_\ell = \alpha_\ell'
        +\frac{\alpha_\ell''-\alpha_\ell'}{2}
         \frac{-(\alpha_\ell''-\alpha_\ell')
                \langle\nabla \phi(x_{\ell}'),d\rangle}
            {\phi(x_{\ell}'')-\phi(x_{\ell}') -(\alpha_\ell''-\alpha_\ell')
             \langle\nabla \phi(x_{\ell}'),d\rangle}
    >\alpha_\ell'.
\end{align}
By the Mean-Value Theorem for derivatives and the second inequality in (\ref{cond_alpha''}), 
\begin{align}\label{sigma}
    0<(1-M_{\ell})^{-1}
    \leq&\ \Big(1-\frac{\langle\nabla \phi(\bar x_{\ell}),d\rangle}
                    {\langle\nabla \phi(x_{\ell}'),D_k\rangle}\Big)^{-1} \nonumber\\
    =&\ \frac{-(\alpha_\ell''-\alpha_\ell')
                \langle\nabla \phi(x_{\ell}'),D_k\rangle}
            {\phi(x_{\ell}'')-\phi(x_{\ell}') -(\alpha_\ell''-\alpha_\ell')
             \langle\nabla \phi(x_{\ell}'),D_k\rangle}
    <\frac{\sigma}{\sigma-\rho},
\end{align}
where $\bar x_{\ell} = x+\bar\alpha_\ell d$ with 
$\bar\alpha_\ell\in [\alpha_\ell', \alpha_\ell'']$ and 
\[
    M_{\ell} 
    = \min_{\alpha\in [\alpha_\ell', \alpha_\ell'']}
    \frac{\langle\nabla \phi(x+\alpha d),d\rangle}
         {\langle\nabla \phi(x_{\ell}'),d\rangle}
    \leq \frac{\langle\nabla \phi(\bar x_{\ell}),d\rangle}
              {\langle\nabla \phi(x_{\ell}'),d\rangle}
    <\frac{\rho}{\sigma}. 
\]
Hence,  if $0<2\rho<\sigma$, we have that
\begin{align}\label{c_ell bound}
    \alpha_\ell'<\alpha_\ell'+ \frac{1}{2(1-M_{\ell})}(\alpha_\ell''-\alpha_\ell')
    < c_\ell <\alpha_\ell'+\frac{\sigma}{2(\sigma-\rho)}(\alpha_\ell''-\alpha_\ell')
    <\alpha_\ell''.
\end{align}

We may shrink $[\alpha_\ell',\alpha_\ell'']$ to $[c_{\ell},\alpha_\ell'']$ or $[\alpha_\ell',c_{\ell}]$, if $\alpha=c_{\ell}$ satisfies (\ref{wolfe1}) or does not. However, if (\ref{wolfe1}) is satisfied, the interval length is 
$
    \alpha_\ell''-c_{\ell}
    \leq \frac{1-2M_{\ell}}{2-2M_{\ell}}(\alpha_\ell''-\alpha_\ell').
$
When $M_{\ell}<0$ and $|M_{\ell}|$ is large, $\frac{1-2M_{\ell}}{2-2M_{\ell}}\approx 1$. The interval shrinking is inefficient in this case. 
To avoid this phenomenon, we slightly modify $c_\ell$ as that with $\eta = \frac{\sigma}{2(\sigma - \rho)}$
\begin{align}\label{tilde_c}
    \tilde c_{\ell} 
    = \max\big\{ c_{\ell},\ \eta \alpha_{\ell}' +(1-\eta)\alpha_{\ell}'' \big\}
    \in(\alpha_\ell',\alpha_\ell'').
\end{align}
Since $\tilde c_{\ell} \geq\eta \alpha_{\ell}' +(1-\eta)\alpha_{\ell}''$ and $c_\ell<\alpha_\ell'+\eta(\alpha_\ell''-\alpha_\ell')$ by (\ref{c_ell bound}), we get
\[
    \alpha_{\ell}''- \tilde c_{\ell} \leq \eta(\alpha_\ell''-\alpha_\ell'),\quad
    \tilde c_{\ell}- \alpha_{\ell}'
      \leq \max\big\{\eta,1-\eta\big\}(\alpha_\ell''-\alpha_\ell')
      = \eta(\alpha_\ell''-\alpha_\ell').
\]
The last equality holds since $\eta>1/2$. Hence, if the Wolfe conditions (\ref{wolfe1}-\ref{wolfe2}) are satisfied for $\alpha = \tilde c_\ell$, we get the required $\alpha_k=\tilde c_\ell$. Otherwise, shrink $[\alpha_\ell',\alpha_\ell'']$ as 
\begin{align}\label{interval_alpha}
    [\alpha_{\ell+1}',\alpha_{\ell+1}'']
    =\left\{\begin{array}{ll}
         [\alpha_\ell', \tilde c_\ell], & \ \mbox{if (\ref{wolfe1}) does not hold for $\alpha = \tilde c_{\ell}$}; \\
         \mbox{$[\tilde c_\ell,\alpha_\ell'']$}, & \ \mbox{otherwise}.
    \end{array}
    \right.
\end{align}
The interval length is significantly decreased as $0<\alpha_{\ell+1}''-\alpha_{\ell+1}'\leq \eta(\alpha_{\ell}''-\alpha_{\ell}')$,
where $\eta<1$ since $2\rho <\sigma$. Hence, $\alpha_{\ell}''-\alpha_{\ell}'\to 0$ as $\ell\to \infty$. 

\begin{lemma}
If $\phi$ is lower bounded and continuously differentiable, an $\alpha = \tilde c_{\ell^*}$ satisfying (\ref{wolfe1}-\ref{wolfe2}) can be obtained within a finite iterations of (\ref{interval_alpha}) if $0< 2\rho <\sigma<1$. 
\end{lemma}
\begin{proof}
If (\ref{wolfe1}-\ref{wolfe2}) do not hold for all $\tilde c_\ell$, the updating rule (\ref{interval_alpha}) yields a sequence of nested intervals $\{[\alpha_\ell',\ \alpha_\ell'']\}$. Since $0< 2\rho <\sigma<1$, the intervals tend to a single point $\alpha_*$ and both $\{x_\ell'\}$ and $\{x_\ell''\}$ tend to $x_* = x+\alpha_* d$. Hence, by (\ref{cond_alpha''}) and the Taylor extension of $\phi(x+\alpha d)$ at $\alpha=\alpha_*$, we get
\begin{align}\label{sigma_d}
    \langle\nabla \phi(x_*),d\rangle = 
    \lim_{\ell\to \infty} \frac{\phi(x_{\ell}'')-\phi(x_{\ell}')}{\alpha_\ell'' - \alpha_\ell'}
    \geq \rho \langle\nabla \phi(x),d\rangle
    > \sigma \langle\nabla \phi(x),d\rangle
\end{align}
since $\langle\nabla \phi(x),d\rangle<0$ and $\rho<\sigma$. However, by (\ref{cond_end}),
$\langle\nabla \phi(x_*),d\rangle\leq\sigma\langle\nabla \phi(x),d\rangle$, a contradiction with (\ref{sigma_d}).
$\hfill\square$
\end{proof}

A good choice of $\alpha_0''$ helps to pursue the minimizer $\alpha_{\min}$. 
Motivated by the above analysis on the estimation of the shrinking rate $\eta_{\ell}$, we suggest the experiential setting 
\begin{align}\label{alpha_0''}
    \alpha_0'' = \min\big\{\alpha = 2^p\eta: \ 
    \mbox{(\ref{wolfe1}) is not satisfied for $\alpha=2^p\eta$ with integer $p\geq 0$} \big\}.
\end{align}
Starting with the initial setting, the interval updating (\ref{interval_alpha}) converges quickly. For instance, we set $\rho = 0.1$ and $\tau = 0.4$, the interval iteration terminates within one or two iterations generally in our experiments. Algorithm 1 gives the details of the procedure for determining an inexact line search $\alpha_k$, given $x_k$, $\phi_k$, $g_k$, the conjugate direction $d_k$.

\begin{algorithm}[t]
\setstretch{1.35}
\caption{An inexact line search satisfying the weak Wolfe conditions} \label{alg:stepsize}
\begin{algorithmic}[1]
    \REQUIRE point $x$, $\phi = \phi(x)$, $g=\nabla\phi(x)$, direction $d$, and parameters $\sigma$, $\rho$, $\varepsilon$
    \ENSURE $\alpha$ satisfying (\ref{wolfe1}-\ref{wolfe2}) within accuracy $\varepsilon$, $x := x+\alpha d$, $\phi(x)$, and $g=\nabla\phi(x)$.
    \STATE Set $\alpha' = 0$, $x' = x$, $\phi' = \phi$, $g' = g$, 
        $\nu = \rho\langle g,d\rangle$. Find the smallest integer $p\geq 1$ such that (\ref{wolfe1}) does not hold for $\alpha =\eta 2^p$, and set $\alpha''=\eta 2^p$.
    \STATE Repeat the following iteration until $\alpha''-\alpha'<\varepsilon$. 
    \STATE \hspace{10pt} 
        Compute $c$ as (\ref{alpha_k}), $\tilde c$ as (\ref{tilde_c}), and 
        $\tilde\phi = \phi(\tilde x)$ at $\tilde x=x+\tilde cd$. 
    \STATE \hspace{10pt} 
        If $\tilde\phi > \phi+\tilde c\nu$, update $(\alpha'',\phi(x''))$ by $(\tilde c,\phi(\tilde x))$ and go to Step 3.
    \STATE \hspace{10pt} 
        Compute $\tilde g = \nabla\phi(\tilde x)$. 
        If $\langle\tilde g, d\rangle \geq \sigma\langle g,d \rangle$, set 
        $x =\tilde x$, $\phi = \tilde\phi$, $g = \tilde g$, and terminate.
    \STATE \hspace{10pt} 
        Otherwise, update $\alpha',\phi',g'$ by $\tilde c$, $\tilde\phi,\tilde g$, respectively.
    \STATE End iteration
\end{algorithmic}
\end{algorithm}

\subsection{Convergence of the modified NCG}

We have two results for the convergence.

\begin{lemma}\label{lma:sdc}
Let $\beta_k$ be defined by (\ref{MMPRP}) and $\mu = \frac{4\nu-1}{4\nu(1+\kappa)}$. Then  
\begin{equation}\label{cond:sdc}
    \langle d_{k+1},g_{k+1} \rangle \leq -\mu \|d_{k+1}\|_2\|g_{k+1}\|_2.
\end{equation}
\end{lemma}
\begin{proof}
We rewrite $\beta_k = \rho_k \tilde\beta_k$ and  
$d_{k+1} =\rho_k\tilde d_{k+1}+(\rho_k-1)g_{k+1}$, where $\rho_k\in[0,1]$,
$\tilde\beta_k$ is the step in (\ref{MPRP}), and $\tilde d_{k+1} = -g_{k+1} +\tilde\beta_k d_k$. 
Since $\langle\tilde d_{k+1},g_{k+1}\rangle \leq (\frac{1}{4\nu}-1) \|g_{k+1}\|^2_2$,
\begin{align*}
    \langle d_{k+1},g_{k+1} \rangle 
    =&\ \rho_k\langle \tilde d_{k+1},g_{k+1} \rangle +(\rho_k-1)\|g_{k+1}\|_2^2\\
    \leq&\ \big(\rho_k(\frac{1}{4\nu}-1)+(\rho_k-1)\big)\|g_{k+1}\|_2^2
    \leq\frac{1-4\nu}{4\nu}\|g_{k+1}\|_2^2.
\end{align*}
By $\beta_k\leq \frac{\kappa\|g_{k+1}\|_2}{\|d_k\|_2}$, we also have that  $\|d_{k+1}\|_2\leq\|g_{k+1}\|_2+\beta_k\|d_k\|_2
\leq (1+\kappa)\|g_{k+1}\|_2$. The inequality (\ref{cond:sdc}) follows
since $\|g_{k+1}\|_2^2\geq \frac{1}{1+\kappa}\|d_{k+1}\|_2\|g_{k+1}\|_2$ and $\frac{1-4\nu}{4\nu}<0$.
$\hfill\square$
\end{proof}

Combining (\ref{wolfe1}), the inequality (\ref{cond:sdc}) guarantees the descent of $\{\phi(x_k)\}$. The NCG with the inexact line search discussed in the previous subsection and the modified step $\beta_k$ given in (\ref{MMPRP}) is convergent if $\phi$ is continuously differentiable and lower bounded. The convergence analysis is slightly different from that for the PRP step in \cite{SY2006}.

\begin{theorem}\label{conv_CG}
Starting with an arbitrary $x_0$, the NCG with inexact line search $\{\alpha_k\}$ satisfying the weak Wolfe condition (\ref{wolfe1}-\ref{wolfe2}) and steps $\{\beta_k\}$ defined in (\ref{MMPRP}) is convergent in the sense that $\{\phi(x_k)\}$ is monotone decreasing and converges and that $\nabla \phi(x_k)\!\to\! 0$. 
\end{theorem}
\proof 
We assume $g_k=\nabla \phi(x_k)\neq 0$ for each $k$  without loss of generalities, and let $s_k = \alpha_kd_k$. By Lemma \ref{lma:sdc},
$\langle g_k,s_{k}\rangle\leq-\mu\|g_k\|\|s_{k}\|\leq 0$. The Wolfe condition (\ref{wolfe1}) gives 
\[
    \phi(x_{k+1})-\phi(x_k)\leq\rho\langle g_k,s_k\rangle
    \leq-\rho\mu\|g_k\|\|s_{k}\|\leq0.
\]
It means that $\{\phi(x_k)\}$ is monotone decreasing. Hence, it is convergent since $\phi$ itself is lower bounded, which also implies that  $\|g_k\|\|s_{\!k}\|\to 0$ by the above inequalities. 

We further show that $\|g_k\|\to 0$. Otherwise, there is a subsequence $\{\|g_{k_i}\|_2\}$ with a positive lower bound. Correspondently, $\|g_{k_i}\|\|s_{\!k_i}\|\to 0$ implies that $\|s_{\!k_i}\|\to 0$. By the Taylor extensions
\[
    \phi(x_{k_i+1}) 
    = \phi(x_{k_i}) + \langle g_{k_i},s_{k_i}\rangle + o\big(\|s_{k_i}\|\big),\ \ 
    \phi(x_{k_i}) 
    = \phi(x_{k_i+1}) - \langle g_{k_i+1},s_{k_i} \rangle + o\big(\|s_{k_i}\|\big),
\]
and the Wolfe condition (\ref{wolfe2}) that gives $\langle g_{k_i+1},s_{k_i}\rangle \geq \sigma \langle g_{k_i},s_{k_i}\rangle$, we have that
$$
    o(\|s_{k_i}\|) 
    = \langle g_{k_i+1},s_{k_i} \rangle - \langle g_{k_i},s_{k_i} \rangle
    \geq (\sigma-1)\langle g_{k_i},s_{k_i}\rangle.
$$
Hence, 
$
    1-\sigma = \frac{o(\|s_{k_i}\|)}{-\langle g_{k_i},s_{k_i}\rangle}
    \leq \frac{o(\|s_{k_i}\|)}{\mu\|g_{k_i}\|\|s_{k_i}\|}\to 0
$ since $\{\|g_{k_i}\|\}$ has a positive lower bound, which implies $\sigma\leq 1$, a contradiction with $\sigma<1$.
$\hfill\square$
\endproof

\begin{algorithm}[t]
\setstretch{1.35}
\caption{Modified nonlinear conjugated gradient (MNCG) method}
\label{alg:ncg}
\begin{algorithmic}[1]
    \REQUIRE initial point $x$, parameters $\epsilon_1$, $\epsilon_2$, $\rho$, $\sigma$, $\nu$, $\kappa$, $\lambda$, and $k_{\max}^{\rm NCG}$.
    \ENSURE an approximate solution $x_*$ of $\min_x \phi(x)$ with the given accuracy
    \STATE Compute $\phi = \phi(x)$, $g = \nabla \phi(x)$, and set $d = -g$.
    \STATE For $k=1,\cdots, k_{\max}^{\rm NCG}$,
    \STATE \hspace{10pt} 
        Save $\phi_{old} = \phi$, and update $(x,\phi,g)$ by \textbf{Algorithm \ref{alg:stepsize}} with searching direct $d$.
    \STATE \hspace{10pt} 
         If $\|g\|_F< \epsilon_1$ and $\phi_{old}-\phi<\epsilon_2$, 
        then set $x^* = x$ and terminate the iteration. 
    \STATE \hspace{10pt} 
        Otherwise, compute $\beta$ by (\ref{MMPRP}) and update $d := -g+\beta d$.
    \STATE End for
\end{algorithmic}
\end{algorithm}

Algorithm \ref{alg:ncg} gives the details of NCG for minimizing a nonlinear function $\phi(x)$. We will use it to solve the exterior point model (\ref{prob:epm}). The algorithm performs very well in our tests. For instance, applying on a symmetric factorization $A = WW^T$ of a completely positive matrix of order 20000 with cp-rank 10, the algorithm can get a CPF with accuracy $10^{-14}$ within 150 iterations and 3 seconds, starting at the identity matrix of order $r$. As a comparison, using the same initial point, the alternative projection method given in \cite{HSS2014} gives an approximate CPF in the accuracy $10^{-11}$, which costs more than 450000 iterations and more than 1500 seconds.

\subsection{Postprocessing}

Generally, a solution $X$ of (\ref{prob:epm}) solve by the modified NCG is not exactly row-orthonormal since the algorithm terminates within a limit accuracy. We can get an approximate CPF with a nonnegative factor $\tilde B = (WX)_+$ truncated from $WX$. 

To improve the accuracy of the approximate CPF, we suggest postprocessing on the solution $X$. That is, find a row-orthonormal matrix $Q\in {\mathbb R}^{r\times r_+}$ nearest to $X$ at first, and then truncate $WQ$ to be a nonnegative $B_+ = (WQ)_+$. This $Q$ can be a solution to the Procrustes problem $\min_{QQ^T = I}\|X - Q\|_F$. That is, $Q = UV^T$ when we have the singular value decomposition $X = U\Sigma V^T$ of $X$, where $U$ is an orthogonal matrix of order $r$ and $V\in {\mathbb R}^{r\times r_+}$ is column-orthonormal. The following estimation gives insight into the improvement. 

Let $N = (WX)_-$ for simplicity, then $(WX)_+ = WX-N$, and 
\begin{align*}
    \|A-(WX)_+(WX)_+^T\|_F 
    &= \|A-(WX)(WX)^T + (WX)N^T+N(WX)^T-NN^T\|_F\\
    &\leq \|A-(WX)(WX)^T\|_F+(2\|WX\|_2+\|N\|_{\infty})\|N\|_F,
\end{align*}
where $\|N\|_{\infty}$ is the largest absolute entry of $N$.
When $X=Q$, it is simplified as
\[
    \|A-(WQ)_+(WQ)_+^T\|_F \leq  \big(2\sqrt{\|A\|_2}+\|N\|_{\infty}\big)\|N\|_F.
\]
This postprocessing may slightly increase the negative component $\|N\|_F$, but it vanishes the term $\|A-(WX)(WX)^T\|_F$, and yields a significant decreasing of the approximate error eventually. In our experiments, we always adopt the postprocessing and take the orthogonal projection $Q$ of a solution $X$ as an eventual output.

\section{Potential issues influencing the CPF}\label{sect:explore}

The CPF was thought to be NP-hard in \cite{GD2018} without proofs, even if the column number of a nonnegative factor is relaxed to be larger than the cp-rank.\footnote{We say $A = BB^T$ is a weak CPF later if the column number of the nonnegative $B$ is larger than the cp-rank $r_{cp}$ of $A$, distinguishing it from the strict CPF whose factor has $r_{cp}$ columns.} That is, one is not able to get an algorithm to compute such a CPF for all completely positive matrices within polynomial time of the matrix order. However, it may be possible to get a good factorization with high accuracy for some completely positive matrices within acceptable time. It is tricky that we know less about what kind of completely positive matrices whose CPF is easy or hard to obtain. 

In this section, we will explore some potential issues that may influence the CPF numerically, implemented by our exterior point method using the modified NCG that is given in the previous section.
We focus on the three issues on the truly existed nonnegative factor $B$ of a completely positive matrix $A$: the distribution of its column norms, its sparsity, and its approximately rank deficiency. It is not clear whether a fixed $A$ has multiple CPFs whose nonnegative factors $B$ have quite different properties on these three issues.\footnote{It is more likely for $A$ with a dense nonnegative factor to have multiple CPFs.} However, we do not find evident differences in our experiments when $A$ has a spare nonnegative factor $B$.

Four kinds of distributions of the column norms of $B$ are considered: constant, linear, convex, or concave. In each set of those $B$'s with the same kind of the column distribution, we also consider the influences of the column number (the cp-rank of $A$), sparsity, and approximate rank deficiency of $B$ to the CPF. Synthetic completely positive matrices are randomly constructed with these properties. Because of the construction, we always have that $r_{cp}(A) = r(A)$ for these matrices. Hence, we set $r_+ =r(A)$. For simplicity, we also fix the order of these synthetic matrices as $n = 200$. A few completely positive matrices with cp-rank larger than rank reported in the literature and the synthetic completely positive matrices in a larger scale ($n=20000$ for example) will be tested in the comparison section given later.

As mentioned in Section \ref{subsect:OPM}, we always normalize the factor $W$ from the symmetric factorization $A = WW^T$ to $\tilde W = D^{-1}W$ with $D = \diag(\|w_1\|_2,\cdots,\|w_n\|_2)$ before its CPF. The postprocessing discussed in the last section is also adopted. That is, we use the orthogonal projection $Q$ as the output and set $\tilde B=(WQ)_+$ as an approximate nonnegative factor of $A$. We measure the factorization accuracy by the relative error
\begin{align}
	\mbox{Error}(\tilde A) = \frac{\|A - \tilde B\tilde B^T\|_F}{\|A\|_F}.
\end{align}

\subsection{Column distribution} \label{subsect:distribution}

Completely positive matrices in the form $A = BB^T$ can be easily constructed by randomly choosing a nonnegative factor $B$ with a given number of columns. Generally, the cp-rank of such a matrix $A$ is also equal to its rank. We consider four sets of $A$'s with the different distributions of column norm sequence $\{b_i\}$ of $B$: One is that with constant $b_i = 1$ for all $i$, and the others have the same form as
\begin{align}\label{b_i}
    b_i = 1-(1-b_{\min})\frac{t_i-t_1}{t_r-t_1} \in [b_{\min},1], \quad i=1,\cdots,r,
\end{align}
where $r$ is the number of columns, $t_i = i^d$ and $d=-10^{-1}$, 1, and 2, respectively. The different values of $d$ determine the different sharp of $\{b_i\}$: convex ($d=-10^{-1}$), linear ($d = 1$), and concave ($d = 2$). The parameter $b_{\min}$ determines how small some of columns of $B$ can be in these three types. 

The four types of $B$'s are constructed as follows. We first choose $\hat B = (\hat b_{ij})$ or order $n\times r$ with entries uniformly distributed in the interval $(0,1)$, and then normalize each column of $\hat B$ to $B = (b_{ij})$ that has a given column norm sequence $\{b_i\}$. That is, the entries are $b_{ij} = \hat b_{ij}\big(b_j/\sqrt{\sum_k \hat b_{kj}^2}\big)$.

\begin{table}[t]
\centering
\resizebox{\textwidth}{!}{ %
\begin{tabular}{|c|@{\ }c@{\ }c@{\ }c@{\ }c|c@{\ }c@{\ }c@{\ }c@{\ }|c@{\ }c@{\ }c@{\ }c|c@{\ }c@{\ }c@{\ }c|}\hline\hline
$r$& \multicolumn{4}{c|}{  5}
   & \multicolumn{4}{c|}{ 10}
   & \multicolumn{4}{c|}{ 15}
   & \multicolumn{4}{c|}{ 20}\\ \hline
\diagbox{\!\!Type\!\!\!}{\!\!$b_{\min}$\!}
   & $10^{-1}$  & $10^{-2}$  & $10^{-3}$  & $10^{-4}$ 
   & $10^{-1}$  & $10^{-2}$  & $10^{-3}$  & $10^{-4}$ 
   & $10^{-1}$  & $10^{-2}$  & $10^{-3}$  & $10^{-4}$ 
   & $10^{-1}$  & $10^{-2}$  & $10^{-3}$  & $10^{-4}$  \\ \hline
constant
   & \multicolumn{4}{c|}{ 100}
   & \multicolumn{4}{c|}{ 45}
   & \multicolumn{4}{c|}{ 98}
   & \multicolumn{4}{c|}{100}\\ 
 linear 
   &  90 &  99 &  99 &  99
   &  98 & 100 & 100 & 100 
   & 100 & 100 & 100 & 100 
   & 100 & 100 & 100 & 100\\ 
  convex 
   &  87 &  97 &  97 &  98 
   &  99 & 100 & 100 & 100 
   & 100 & 100 & 100 & 100 
   & 100 & 100 & 100 & 100\\ 
concave 
   &  92 & 100 & 100 & 100 
   &  97 &  99 &  98 &  99 
   &  99 & 100 & 100 & 100 
   & 100 & 100 & 100 & 100\\ 
\hline\hline
\end{tabular}
}
\caption{Percentage of CPF by EPM achieving accuracy $\varepsilon=10^{-13}$, dense factors, and $n = 200$}
\label{tab:dense}
\end{table}

Table \ref{tab:dense} lists the percentages of tested matrices whose CPF can be successfully obtained by our exterior point method (EPM) with the accuracy Error$(\tilde A)<\varepsilon = 10^{-13}$ among 1000 repeats in each type of column norm distribution with fixed $b_{\min}$ and $r$.\footnote{For constant $\{b_i\}$, the percentages are that among 4000 tests.} The parameters are also fixed for all the tests as
\begin{align}\label{parameter setting}
    \lambda = 0.1,\ \rho = 0.1, \ 
    \sigma = 0.4,\  \nu = 1, \ \kappa = 1000, \ 
    \epsilon_1 = 10^{-8}, \ \epsilon_2 = 10^{-32}.
\end{align}
In this experiment, starting at a randomly chosen orthogonal matrix, our algorithm works very well for almost all the cases, except the constant case with $r = 10$. There are only about 45\% matrices of rank $r=10$ whose factorization errors can achieve the accuracy $10^{-13}$ and about 53\% matrices have factorization errors larger than $10^{-4}$. Similar phenomenon also occurs when the matrix size $n$ is larger. The CPF is not difficult for matrices with a dense factor $B$ whose column norms are distributed linearly, convexly, or concavely, except in a spacial case shown in the following phenomenon. 

\vspace{10pt}\noindent 
Phenomenon A. {\it For a dense and full rank $B\geq 0$ with approximately equal column norms, if it has a special column number depending on its row number, the CPF of $A = BB^T$ is relatively difficult than others.}
\vspace{10pt}

It is a bit puzzling. A similar phenomenon also occurs when $B$ is sparse. We will show it in the next subsection.

\begin{table}[t]
\centering
\resizebox{\textwidth}{!}{ 
\begin{tabular}{|r|cccccccc|cccccccc|}\hline\hline
\!\!Type\! & \multicolumn{8}{c|}{constant}
     & \multicolumn{8}{c|}{linear, $b_{\min}=0.1$}\\\hline
\diagbox{\!$s$\!}{$r$} 
      &   6 &   8 &  10 &  12 &  14 &  16 &  18 &  20
      &   6 &   8 &  10 &  12 &  14 &  16 &  18 &  20\\\hline
  1\% & 100 &  77 &  46 &  47 &  90 & 100 & 100 & 100
      &  91 &  92 &  96 &  96 & 100 & 100 & 100 & 100 \\
  4\% &  95 &  61 &  19 &  21 &  82 &  99 & 100 & 100
      &  53 &  31 &  52 &  82 & 100 & 100 & 100 & 100 \\
  7\% &  97 &  67 &  22 &   3 &  52 & 98  & 100 & 100
      &  54 &  49 &  22 &   7 &  58 & 90  & 100 & 100 \\
10\% & 98    & 73    & 19    & 0     & 5     & 67    & 98    & 100   & 66    & 57    & 45    & 26    & 5     & 36    & 93    & 100 \\
13\% & 99    & 78    & 28    & 0     & 0     & 10    & 76    & 100   & 63    & 55    & 56    & 42    & 8     & 0     & 21    & 82 \\
16\% & 99    & 88    & 30    & 3     & 0     & 0     & 7     & 80    & 61    & 53    & 52    & 46    & 31    & 3     & 0     & 22 \\
19\% & 100   & 94    & 60    & 4     & 0     & 0     & 0     & 9     & 53    & 60    & 56    & 46    & 50    & 22    & 2     & 0 \\
22\% & 99    & 96    & 66    & 18    & 0     & 0     & 0     & 0     & 62    & 60    & 54    & 63    & 50    & 52    & 17    & 0 \\
25\% & 100   & 100   & 89    & 36    & 0     & 0     & 0     & 0     & 71    & 54    & 57    & 55    & 49    & 63    & 43    & 6 \\
38\% & 100   & 99    & 97    & 77    & 15    & 0     & 0     & 0     & 71    & 58    & 58    & 67    & 59    & 55    & 52    & 23 \\
31\% & 100   & 100   & 99    & 88    & 55    & 6     & 0     & 0     & 77    & 62    & 46    & 61    & 53    & 55    & 50    & 28 \\
34\% & 100   & 100   & 99    & 100   & 85    & 42    & 5     & 0     & 73    & 65    & 57    & 59    & 58    & 65    & 55    & 44 \\
37\% & 100   & 100   & 100   & 97    & 95    & 81    & 34    & 1     & 90    & 70    & 57    & 64    & 67    & 63    & 55    & 46 \\
40\% & 100   & 100   & 100   & 100   & 98    & 91    & 80    & 30    & 92    & 74    & 69    & 60    & 58    & 64    & 59    & 48 \\
\hline\hline
\end{tabular}
}
\caption{Percentage of CPF achieving accuracy $\varepsilon=10^{-13}$ depending on $s=s(B)$ and $r = r_{cp}(A)$ ($n=200$)}
\label{tab:sparse}
\end{table}

\subsection{Sparsity} \label{subsect:sparsity}

Besides the distribution of the column norms, the sparsity of factor $B$ is another issue that may affect the CPF. To show this phenomenon, we modify the construction in the previous subsection by vanishing partial smallest entries of $\hat B$ with a given sparsity in percentage, without modifications on the normalization for having a special column distribution. Generally, the sparsity strategy does not change the rank and cp-rank if the sparsity is not large. However it affects the difficulty of CPF. \footnote{If the sparsity is large, some rows of $\hat B$ may be zero. These rows should be deleted and the size $n$ of the resulting $A$ is slightly reduced.}

Table \ref{tab:sparse} lists the percentages of tested matrices whose CPF errors are smaller than  $10^{-13}$ among 100 repeats with the constant or linear distribution of the column norms (the results for concave or convex distribution can be referred in Table \ref{tab:rank-deficient}). 
It shows that the hardness of CPF does not monotonically depend on the rank/cp-rank or the sparsity of a nonnegative factor $B$. Very interestingly, the hardness is tightly related to a special rank-sparsity boundary.

\vspace{10pt}\noindent
Phenomenon B. {\it For each column norm distribution, there is a special rank-sparsity boundary such that the closer to this boundary the rank-sparsity pair of $B$ is, the more difficult the CPF of $A=BB^T$ is.}

\subsection{Approximate cp-rank deficiency} \label{subsect:deficiency}

A small value of $b_{\min}$ means an approximately deficient cp-rank of the completely positive matrix $A$ tested in the previous subsections. This is because we can rewrite $A = \sum_i^r d_iu_iu_i^T$ with nonnegative unit vectors $\{u_i\}$ and $d_r = b_{\min}^2$. For instance, if $b_{\min} = 10^{-4}$, then $d_r = 10^{-8}$. Clearly, deleting the last column of $B$ just slightly modifies $A$ to be a completely positive matrix $\hat A=\sum_i^{r-1} d_iu_iu_i^T$ that is very close to $A$ with $\|A-\hat A\|_F = d_r$, but $\hat A$ has a smaller cp-rank. The approximate deficiency does not affect the CPF very much if the nonnegative factor $B$ is dense, as shown in Table \ref{tab:dense}. However, as the sparsity is increased, the influence of approximate cp-rank deficiency to the CPF is more and more evident.

\begin{table}[t]
\centering
\resizebox{\textwidth}{!}{ 
\begin{tabular}{|r|cccc|cccc|cccc|cccc|}\hline\hline
\!\!Type\! & \multicolumn{8}{c|}{convex} & \multicolumn{8}{c|}{concave}\\\hline
\!\!$b_{\min}$ 
      & \multicolumn{4}{c|}{$10^{-1}$} & \multicolumn{4}{c|}{$10^{-2}$}
      & \multicolumn{4}{c|}{$10^{-1}$} & \multicolumn{4}{c|}{$10^{-2}$}\\\hline
\diagbox{\!$s$\!}{$r$} 
      &   8 &  10 &  12 &  14 &   8 &  10 &  12 &  14
      &   8 &  10 &  12 &  14 &   8 &  10 &  12 &  14\\\hline
     1\%& 87    & 98    & 99    & 100   & 100    & 100   & 100   & 100  
        & 97    & 94    & 97    & 100   & 98     & 100    & 99   & 98 \\
     4\%& 29    & 50    & 87    & 99    & 21    & 64    & 93    & 98 
        & 41    & 51    & 69    & 92    & 23    & 62    & 76    & 96 \\
     7\%& 50    & 27    & 12    & 55    & 0     & 1     & 8     & 52 
        & 41    & 23    & 10    & 47    & 0     & 2     & 5     & 44 \\
    10\%& 63    & 52    & 11    & 3     & 1     & 0     & 0     & 1    
        & 54    & 47    & 9     & 2     & 0     & 0     & 0     & 1 \\
    13\%& 73    & 63    & 44    & 8     & 2     & 0     & 0     & 0     
        & 51    & 49    & 35    & 0     & 0     & 0     & 0     & 0 \\
    16\%& 76    & 57    & 56    & 33    & 8     & 0     & 0     & 0   
        & 59    & 49    & 43    & 19    & 2     & 0     & 0     & 0 \\
    19\%& 69    & 74    & 73    & 54    & 6     & 0     & 0     & 0  
        & 55    & 56    & 50    & 35    & 2     & 0     & 0     & 0 \\
    22\%& 69    & 66    & 67    & 59    &22     & 1     & 0     & 0  
        & 52    & 64    & 46    & 49    & 0     & 0     & 0     & 0 \\
    25\%& 63    & 73    & 65    & 64    &33     & 3     & 1     & 0  
        & 56    & 47    & 57    & 46    & 3     & 1     & 0     & 0 \\
    28\%& 72    & 69    & 74    & 67    &39     & 4     & 0     & 0  
        & 62    & 53    & 54    & 53    & 2     & 1     & 0     & 0 \\
    31\%& 78    & 77    & 74    & 68    &41     & 9     & 2     & 0  
        & 58    & 65    & 47    & 44    &14     & 3     & 0     & 0 \\
    34\%& 75    & 69    & 70    & 75    &53     &18     & 0     & 0 
        & 68    & 59    & 56    & 56    &22     & 2     & 0     & 0 \\
    37\%& 72    & 73    & 72    & 73    &59     &36     & 2     & 0  
        & 62    & 60    & 55    & 48    &40     & 0     & 0     & 0 \\
    40\%& 82    & 73    & 78    & 72    &53     &32     & 7     & 0 
        & 65    & 53    & 55    & 44    &49     & 7     & 1     & 0 \\
    \hline\hline
\end{tabular}
}
\caption{Dependence of successful CPF ($\varepsilon = 10^{-13}$) on approximate cp-rank deficiency}
\label{tab:rank-deficient}
\end{table}

Table \ref{tab:rank-deficient} illustrates the phenomenon. We test two values of parameter $b_{\min}$ as $10^{-1}$ and $10^{-2}$ in the construction of completely positive matrices with concave distribution or convex distribution of column norms. As we decrease $b_{\min}$ from $10^{-1}$ to $10^{-2}$, the percentage of successful CPF is decreased evidently when the sparsity of $B$ is not ignorable,\footnote{In the relatively dense case, the percentage of successful factorization is slightly increased as $b_{\min}$ becomes smaller. This phenomenon will be explained in the next section.} especially nearby the rank-sparsity boundary mentioned in Phenomenon B. We summarize it as

\vspace{10pt}\noindent
Phenomenon C. {\it Approximate cp-rank deficiency significantly aggravates the CPF nearby the special rank-sparsity boundary in Phenomenon B.}
\vspace{10pt}

One explanation is that there are many rows of the sparse nonnegative factor alive at the boundary of the nonnegative cone $R_+^{r_+}$ and it is hard to accurately pursue these rows.

\section{Improvements}\label{sect:improvement}

In this section, we consider two strategies to increase the possibility of getting a CPF with acceptable accuracy. One is the relaxation of cp-rank, in which the column number of factor $B$ is relaxed from the exact cp-rank $r_{cp}$ to a larger integer $r_+$. Our algorithm works in this case without any modifications, just starting at an initial matrix, for example, an arbitrarily chosen row-orthonormal matrix, of order $r\times r_+$. The second one is a fast restart strategy. We will give a new stopping criterion based on Theorem \ref{thm:fdf} to reduce the computational cost when the restart strategy is adopted.

\subsection{Weak CPF}

By {\it weak} CPF, we mean a symmetric factorization $A = \tilde B\tilde B^T$ with a nonnegative factor $\tilde B$ with a relaxed column number $r_+$ larger than the cp-rank of $A$, distinguishing it with the {\it strict} CPF with strict $r_{cp}$ columns in its nonnegative factor. Relative to the strict factorization, the weak problem is a bit easier if $A$ is not approximately cp-rank deficient. A rough but reliable explanation is that increasing the column number can enlarge the solution set, and most rows of a solution $\tilde B$ are far from the boundary of the nonnegative cone in a higher dimensional space.

\begin{figure}[t]
\centering
\includegraphics[width = 0.25\linewidth, height = 0.3 \linewidth]{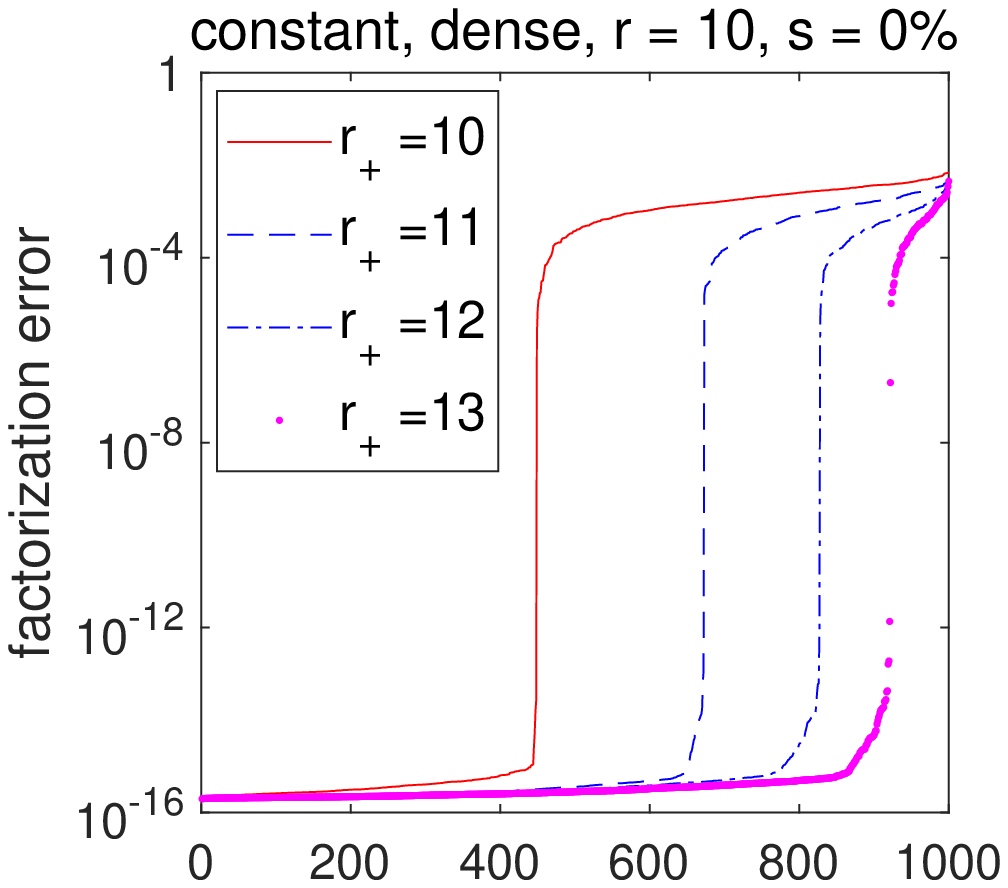}
\hspace{-9pt}
\includegraphics[width = 0.25\linewidth, height = 0.3 \linewidth]{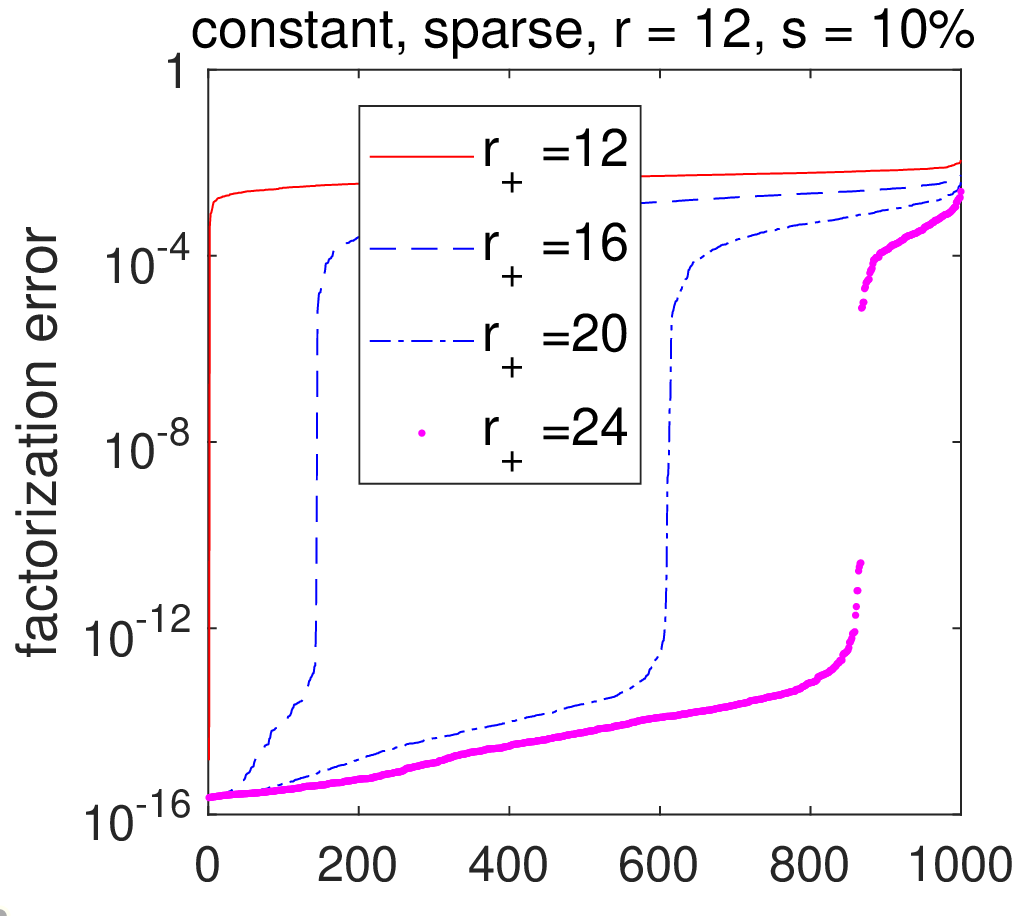}
\hspace{-9pt}
\includegraphics[width = 0.25\linewidth, height = 0.3 \linewidth]{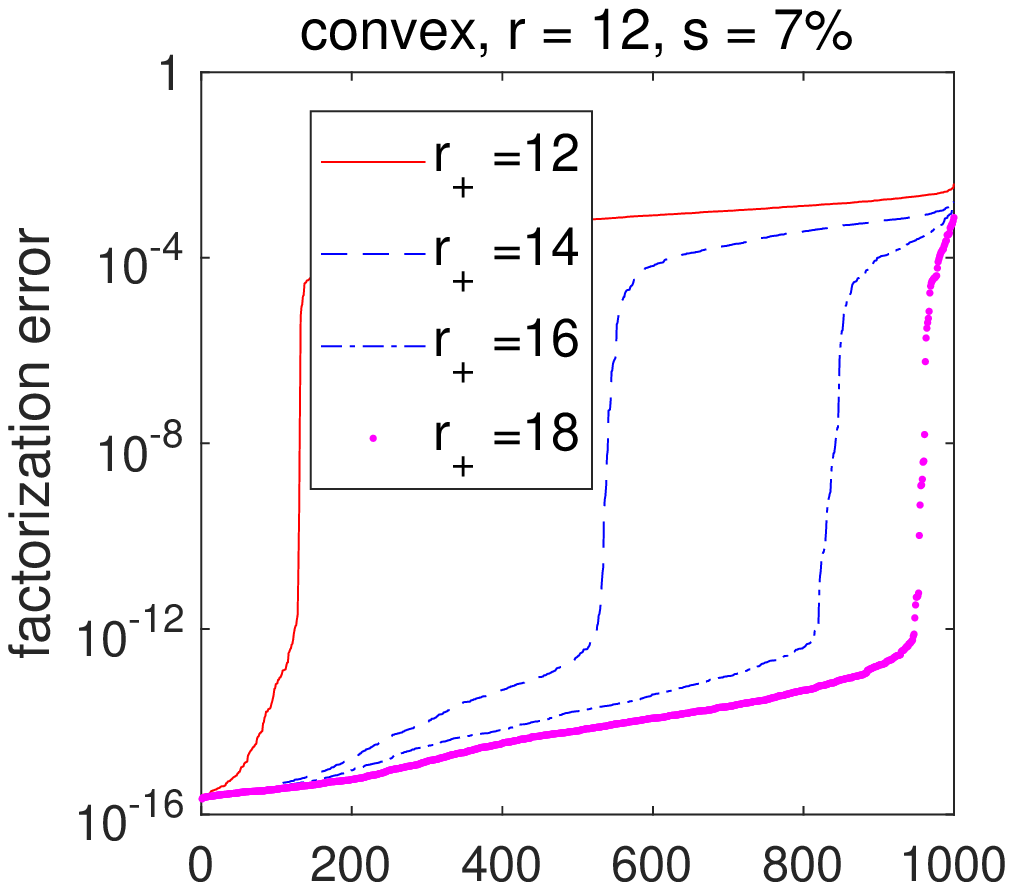}
\hspace{-9pt}
\includegraphics[width = 0.25\linewidth, height = 0.3 \linewidth]{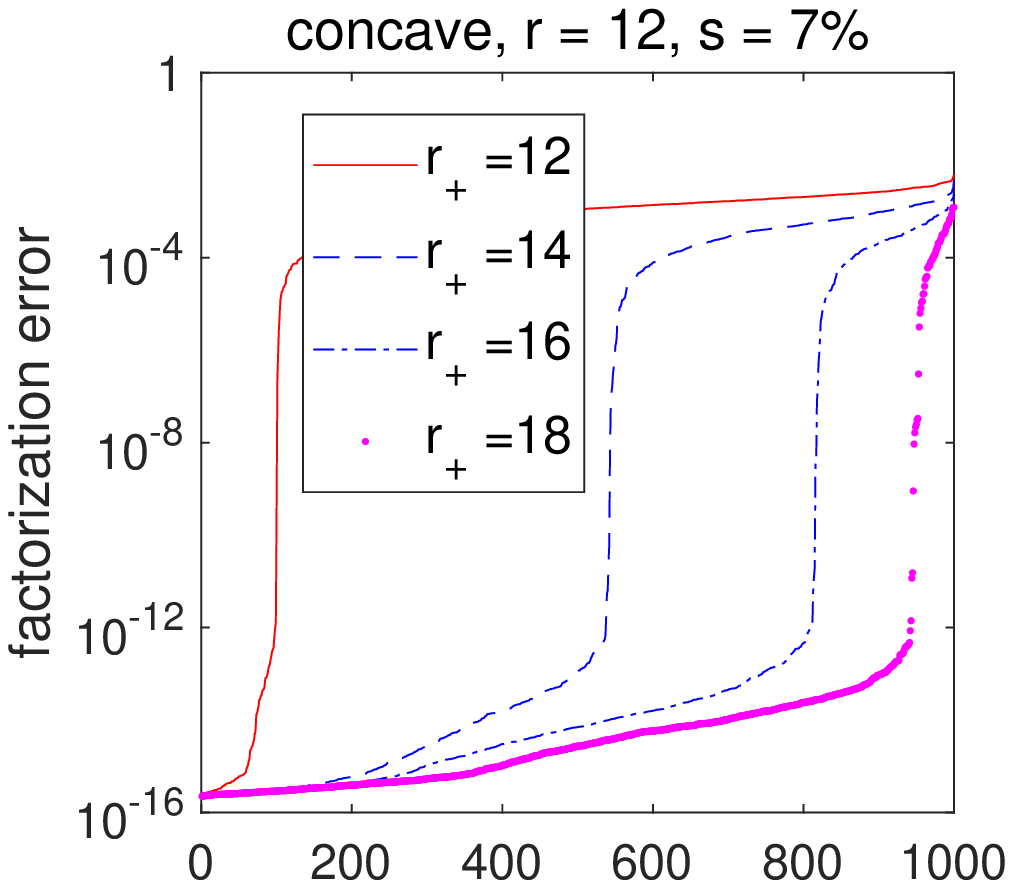}
\caption{Weak CPF errors of 1000 tests on each of four matrices without approximate cp-rank deficiency: constant-type dense factor ($r = 10$), sparse factors ($r = 12$) in the types 'constant' ($s(B) = 10\%$), 'convex' and 'concave' ($b_{\min} = 10^{-1}$, $s(B) = 7\%$), respectively}
\label{fig:cprank}
\end{figure}

\begin{figure}[t]
\centering
\includegraphics[width = 0.32\linewidth, height = 0.3 \linewidth]{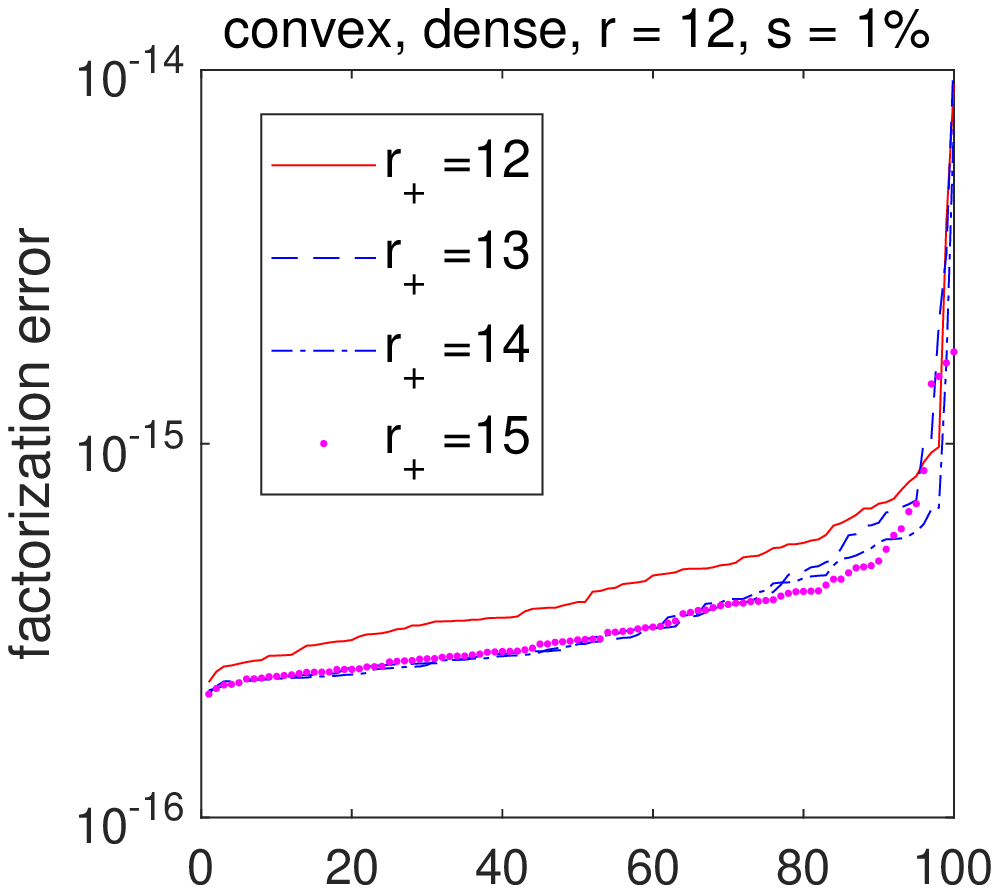}
\hspace{-9pt}
\includegraphics[width = 0.32\linewidth, height = 0.3 \linewidth]{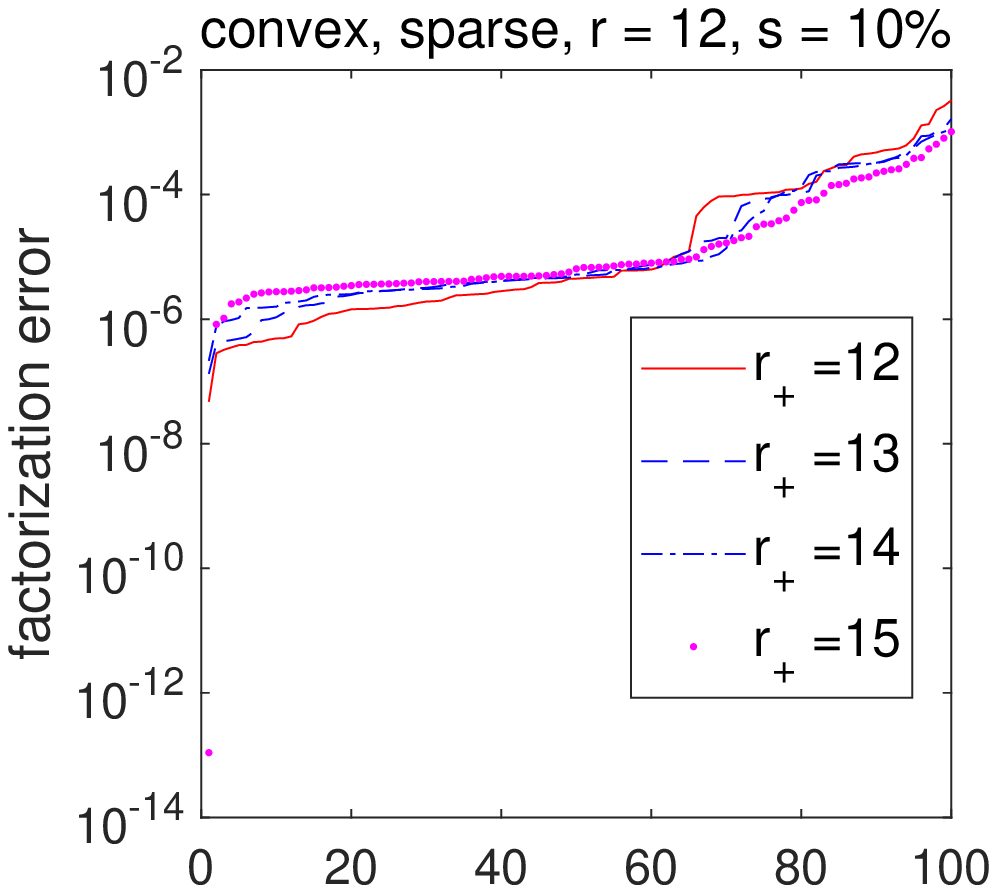}
\hspace{-9pt}
\includegraphics[width = 0.32\linewidth, height = 0.3 \linewidth]{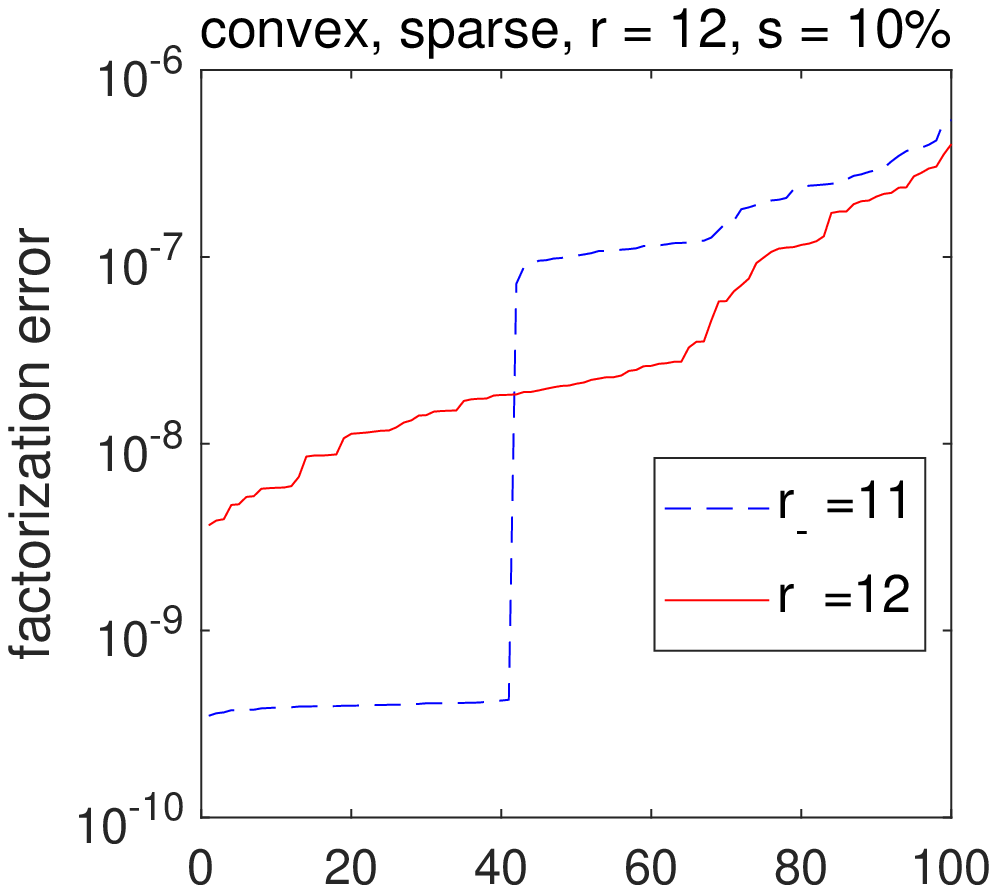}
\caption{Weak CPF errors of 100 tests on two matrices with approximate deficient cp-rank ($b_{\min} = 10^{-4}$) with dense $B$ (left) or sparse $B$ (middle). The left one illustrates the improvement of approximate CPF on a matrix with sparse $B$ like that shown in the middle panel}
\label{fig:relaxedCP_deficiency}
\end{figure}

The relaxation strategy can significantly decrease the hardness of strict CPF on matrices without cp-rank deficiency, whatever $B$ is dense or not. 
The left two panels of Figure \ref{fig:cprank} illustrate the phenomenon on four matrices whose strict CPF are difficult as shown in Tables \ref{tab:dense} and \ref{tab:sparse}. One of the matrices has a dense factor ($r = 10$) with constant-type column norms, and the other three matrices have sparse factors ($r = 12$) with constant-, convex-, or concave-type, and $s(B) = 10\%, 7\%, 7\%$, respectively. For each of them, we repeat the factorization for 1000 times, starting at an arbitrarily chosen row-orthonormal matrix of order $r\times r_+$. In the dense case, due to the increased column number $r_+$, the improvement is significant in the sense that it is more likely to get a weak CPF within a higher accuracy. In the sparse case, the algorithm cannot get a strict CPF with accuracy smaller than $10^{-4}$ in these repeated tests. However, if we increase the column number from $r_{cp} = 12$ to $r_+=18$ (or $r_+ = 24$ for the 'constant'-type), the factorization accuracy of weak CPF can be increased to $10^{-12}$ on about 900 tests among the 1000 tests. 

However, the improvement of weak CPF is limited when the matrix has an approximately deficient cp-rank. As the strict CPF, the weak CPF also works well if the factor $B$ is dense. We illustrate the performance in the left panel of Figure \ref{fig:relaxedCP_deficiency}. The improvement exists but is very slight. It partially explains the slightly higher percentage listed in the first row of Table \ref{tab:rank-deficient}, corresponding to smaller $b_{\min}$. However, the relaxation strategy may lose its efficiency when $A$ is approximately cp-rank deficient and $B$ is sparse, as illustrated in the middle panel of Figure \ref{fig:relaxedCP_deficiency}.

To partially address the difficulty when $A$ is approximately cp-rank deficient and has a sparse factor $B$, we may use the CPF of its lower-rank modification $\tilde A = \tilde W\tilde W^T$ as an approximate CPF of the original $A$ if $\tilde A$ is also nonnegative, where $\tilde W$ is the $W$ whose small columns are deleted. That is, apply the algorithm on $\tilde W$ with a column number $r_-$ smaller than $r_{cp}$. The right panel of Figure \ref{fig:relaxedCP_deficiency} shows the improvement when the strategy of column shrinking is adopted. There are about more than 40\% tests, in which the factorization error can be reduced. However, because of the perturbation existed in the input $\tilde W$, the factorization error to the original $A$ cannot be smaller than $b_{\min}^2$ in scale theoretically.

\subsection{Restart EPM}

Basically, convergence behaviour of the exterior point iteration may depend on the initial point. There are two extreme phenomena that may occur in the iteration: 
\begin{itemize}
    \item The CPF, whatever it is strict or weak, is easy for some matrices -- the iteration algorithm converges quickly, starting at almost any orthogonal matrix.
    \item The CPF is difficult for some matrices -- starting at almost all points, the iteration algorithm always drops to a local minimum quickly or converges to a global optimal solution very slowly. 
\end{itemize} 
\noindent
In other cases, the convergence may depend on the initial point significantly.

\begin{algorithm}[t]
\setstretch{1.35}
\caption{Restart exterior point method (REPM) for the CPF}
\label{alg:REPM}
\begin{algorithmic}[1]
  \REQUIRE {$W$, $r_{cp}$, $\rho$, $\sigma$, $\nu$, $\kappa$, $\lambda$, $\epsilon_f$, 
     $\epsilon_{d\!f}$, $k_{\max}^{\rm NCG}$, and $k_{\rm total}^{\rm REPM}$.}
  \ENSURE Row-orthonormal matrix $Q$ of order $r\times r_{cp}$.
  \STATE Arbitrarily choose a row-orthonormal $X_0\in{\mathbb R}^{r\times r_{cp}}$,     and set $k=0$.
  \STATE Starting with $X_0$, run Algorithm 2 within at most $k_{\max}^{\rm NCG}$  
     iterations until (\ref{opt_local}) or (\ref{opt_global}) is satisfied, 
     and count the iteration number $k_{NCG}$ in this NCG procedure. 
  \STATE Update $k:=k+k_{NCG}$. If $k\geq k_{\rm total}^{\rm REPM}$, terminate the restart procedure. 
  \STATE If (\ref{opt_local}) is satisfied or $k_{\max}^{\rm NCG}$ is achieved, reset $X_0 = -X$, and go to Step 2.
  \STATE If (\ref{opt_global}) holds, continue the iteration of Algorithm 2 until $f(X_k)-f(X_{k-1})<\epsilon_{d\!f}$.
\end{algorithmic}
\end{algorithm}

One may repeatedly test variant initial points to increase the probability of successful CPF. Let $p$ be the probability of an algorithm getting CPF with a given accuracy for fixed $A$, starting at an arbitrarily chosen point. If we take an additional test as soon as the first one fails, the probability is increased to $p+(1-p)p = 1-(1-p)^2$. Generally, the probability is $1-(1-p)^k$ within at most $k$ times of tests. However, it may also cost $k$ times of the original computational cost statistically. If it happens, the cost may be unacceptably expensive if $k$ is large. 

Fortunately, Theorem \ref{thm:fdf} implies the possibility of quickly checking whether the iteration can converge to a global optimizer or not. Clearly, quick checking in time whether the iteration is more likely to converge locally can avoid the unnecessary computational cost, and combining with a restart strategy, it can make the algorithm more competitive for solving the CPF problem. In this subsection, we give a practical approach for restarting the algorithm EPM.

Theorem \ref{thm:fdf} or Eq. (\ref{dff}) shows two extremely different behaviours of the ratio function $\frac{\|\nabla f(\tilde X)\|_F}{f(\tilde X)}$ when $\tilde X$ gets close to a stationary point of $f$. Motivated by this observation, we adopt the two criteria for charging whether the global convergence can be expectant:
\begin{itemize}
    \item Not global optimum. If the current iterative point $X$ satisfies
    \begin{align}\label{opt_local}
        \|\nabla f(X)\|_F<\epsilon_1,\quad f(X)>\|\nabla f(X)\|_F,
    \end{align}
    $X$ is at least a stationary point approximately because of the first inequality, and meanwhile, the second condition means that $X$ is far from a global minimizer.
    \item Global optimum. We obtain a global optimal solution $X$ in a high accuracy if 
    \begin{align}\label{opt_global}
        \|\nabla f(X)\|_F< \epsilon_2,\quad f(X)<\epsilon_f
    \end{align}
\end{itemize}
with $\epsilon_f\ll\epsilon_2\ll \epsilon_1$. For example, $\epsilon_f = 10^{-24}$, $\epsilon_2 = 10^{-13}$, and $\epsilon_1=10^{-3}$, as did in our experiments.

Therefore, if the condition (\ref{opt_local}) is satisfied, we think that the iteration is much likely to drop into a local minimum. In this case, we have to restart the algorithm at a new point. To avoid turning back to the same local minimum, the new starting point should be far from the current one. A simple approach is to set the opposite one $-X$ of $X$ as a new initial point.
If (\ref{opt_global}) is satisfied, the iterative point is approaching at a global minimizer. It is suggested to continue the iteration until convergence in a high accuracy or the restriction on total iteration number $k_{\rm total}$ is achieved. Details of this approach are given in Algorithm \ref{alg:REPM}.

\begin{table}[t]
\centering
\resizebox{\textwidth}{!}{ 
\begin{tabular}{|c|c|c@{\, \, }cccc@{\, \, }c@{\, \, }c|
                 c@{\, \, }cccccc|}\hline\hline
 &\diagbox{$\varepsilon$}{$s$}
 &  1\% &  4\% &  7\% & 10\% & 13\% & 16\% & 19\%
 &  1\% &  4\% &  7\% & 10\% & 13\% & 16\% & 19\%\\\hline
\multirow{4}{*}{\rotatebox{90}{No restart}}
 && \multicolumn{7}{c|}{constant}
  & \multicolumn{7}{c|}{linear}\\\cline{3-16}
 & 1e-12
   & 47 & 21 &  3 &  0 &  0 &  3 &  4
   & 97 & 83 &  7 & 26 & 42 & 46 & 46\\
 & 1e-13
   & 47 & 21 &  3 &  0 &  0 &  3 &  4
   & 95 & 82 &  7 & 26 & 42 & 46 & 46\\
 & 1e-14
   & 47 & 19 &  3 &  0 &  0 &  3 &  4
   & 95 & 80 &  7 &  2 &  5 & 10 & 14\\
   \hline
\multirow{3}{*}{\rotatebox{90}{Restart}}
 & 1e-12
   & 100 & 100 & 77 & 64 & 84 & 99 & 100
   & 100 & 100 & 96 & 100 & 100 & 100 & 100\\
 & 1e-13
   & 100 & 100 & 76 & 64 & 84 & 99 & 100
   & 100 & 99 & 84 & 99 & 100 & 100 & 100\\
 & 1e-14
   & 97 & 96 & 71 & 64 & 84 & 99 & 100
   & 98 & 93 & 56 &  6 & 11 & 17 &  17\\\hline
\multirow{4}{*}{\rotatebox{90}{No restart}}
 && \multicolumn{7}{c|}{convex}
  & \multicolumn{7}{c|}{concave}\\\cline{3-16}
 & 1e-12
   & 99 & 88 & 13 &  11 & 44 & 56 & 73
   & 97 & 71 & 12 &  9 & 35 & 43 & 50\\
 & 1e-13
   & 99 & 87 & 12 &  11 & 44 & 56 & 73
   & 97 & 69 & 10 &  9 & 35 & 43 & 50\\
 & 1e-14
   & 99 & 83 & 11 &  5 & 33 & 48 & 58
   & 96 & 68 & 10 &  1 & 1 & 8 & 2 \\\hline
\multirow{3}{*}{\rotatebox{90}{Restart}}
 & 1e-12
   & 100 & 100 & 87 & 98 & 100 & 100 & 100
   & 100 & 100 & 98 & 100 & 99 & 99 & 100\\
 & 1e-13
   & 100 & 99 & 76 & 98 & 100 & 100 & 100
   & 100 & 99 & 83 & 99 & 99  & 99  & 100\\
 & 1e-14
   & 99 & 91 & 44 &  23 &  63 & 79 &  88
   & 99 & 94 & 52 &  2 &  5 &  11 & 10\\\hline
\hline
\end{tabular}
}
\caption{Percentage of successful CP factorization with $\varepsilon$, using the restart strategy, $r = 12$, $b_{\min} = 0.1$}
\label{tab:restart}
\end{table}

To show the efficiency of REPM, we test matrices of order $n=200$ with rank (cp-rank) $r=12$, constructed in the four types ($b_{\min} = 0.1$) as before. The sparsity $s(B)$ varies from $1\%$ to $19\%$. As shown in Tables \ref{tab:sparse} and \ref{tab:rank-deficient}, the CPF is difficult for most of the constructed matrices. 50 matrices are constructed for each type and each sparsity, and totally 1400 matrices are tested. For each matrix, we run 10 times of Algorithm REPM starting at a randomly chosen orthogonal matrix. The parameters are set 
as in (\ref{parameter setting}) for $\rho$, $\sigma$, $\nu$, $\kappa$, $\lambda$, and 
\begin{align}\label{parameter setting 2}
    \epsilon_1=10^{-3}, \quad
    \epsilon_2 &= 10^{-13}, \quad
    \epsilon_f = 10^{-24}, \quad
    \epsilon_{d\!f} = 10^{-32},\\
    k_{\max}^{\rm NCG} &= 50000,\quad
    k_{\rm total}^{\rm REPM} = 500000.
\end{align}

Table \ref{tab:restart} show the percentages of CPF achieving at a given accuracy $\varepsilon$ obtained by Algorithm REPM among 500 tests for each sparse setting and each type of column norm distribution. The restart strategy can significantly increase the possibility of obtaining a CPF with a high accuracy. 
For example, the original algorithm almost always fails to get a CPF in high accuracy for constant type matrices with sparsities varying from $7\%$ to $19\%$. The factorization via the restart algorithm is very successful in most of the tests.

It should be pointed out that the number of restarts or the computation time depends on the tested matrix. Roughly speaking, the larger of the restart number is, the harder of its CPF is. However, a large number of restarts does not mean an expensive computational cost since a quick restart may occur. For example, on a constant type matrix with $s = 4\%$, REPM needs 12705 iterations within 19.2 times of restarts, costing 0.64 seconds on average. However, for the same type matrix with $s = 19\%$, REPM requires a smaller number of iterations (9845.6) within more restarts (40.2) and less time (0.26 seconds). Statistically, the computational cost also depends on the sparsity when the rank/cp-rank is fixed. 
To show this phenomenon, in Table \ref{tab:restart_NT}, we list the average values of the total iteration numbers, number of restarts, and computational time in second. In the experiment with fixed $r = 12$, the computational cost is relative expensive when $s(B)$ is $7\%$ or $10\%$, compared with other cases.

\begin{table}[t]
\centering
\resizebox{\textwidth}{!}{ 
\begin{tabular}{|c|c|c@{\, \, }cccc@{\, \, }c@{\, \, }c|
                 c@{\, \, }cccccc|}\hline\hline
Type & $s(B)$  &  1\% &  4\% &  7\% & 10\% & 13\% & 16\% & 19\%\\\hline
\multirow{3}{*}{constant}
 & \# of iter &  5608.8  & 12705.2  & 191275.9  & 260210.4  & 172455.1  & 61630.5  & 9845.6 \\
 & \# of rest & 4.6   & 19.2  & 537.7  & 1084.6  & 785.8  & 273.1  & 40.2\\
 & Time(s) & 0.16  & 0.64  & 5.92  & 6.96  & 4.56  & 1.62  & 0.26\\\hline
\multirow{3}{*}{linear}
 & \# of iter & 1116.0  & 4582.4  & 99120.8  & 26103.1  & 8391.4  & 5650.8  & 4737.0 \\
 & \# of rest & 1.0   & 1.5   & 79.6  & 31.8  & 8.9   & 3.1   & 2.2\\
 & Time(s) & 0.03  & 0.45  & 2.99  & 0.70  & 0.22  & 0.14  & 0.12 \\\hline
\multirow{3}{*}{convex}
 & \# of iter & 958.0  & 5635.3  & 167856.5  & 66521.3  & 10671.0  & 5370.8  & 3948.1\\
 & \# of rest & 1.0   & 1.2   & 101.4  & 73.7  & 11.5  & 4.4   & 2.6\\
 & Time(s) & 0.03  & 0.46  & 5.76  & 1.81  & 0.28  & 0.14  & 0.10\\\hline
\multirow{3}{*}{concave}
 & \# of iter & 1271.1  & 2923.3  & 98776.6  & 31673.3  & 20095.9  & 16224.4  & 5613.4\\
 & \# of rest &  1.1   & 1.4   & 111.9  & 65.0  & 20.5  & 17.3  & 3.3\\
 & Time(s) & 0.04  & 0.08  & 4.52  & 0.87  & 0.54  & 0.42  & 0.14\\\hline
\hline
\end{tabular}
}
\caption{REPM ($r = 12$, $b_{\min} = 0.1$): average values of total iteration number, restart number and time}
\label{tab:restart_NT}
\end{table}

\section{Comparisons}\label{sect:comparison}

In this section, we compare our exterior point method with four state-of-art algorithms: two alternative projection methods given in \cite{HSS2014} (marked as  AP-H) and \cite{GD2018} (marked as  AP-G), the coordinate descending method given in \cite{VGL2016} (marked as CD), and the alternative nonnegative least squared method given in \cite{KYP2015} (marked as ANLS).\footnote{It was reported in \cite{VGL2016} that the Newton method \cite{KDP2012} cannot beat the coordinate descending method. We omit the comparison with the Newton method.} The algorithms and their parameter setting are briefly described below.

The AP-H solves $\min_{B\geq 0,QQ^T = I}\|B-WQ\|_F$ via optimizing $\|B_{k-1}-WQ\|_F$ to get $Q_k$ and optimizing $\|B-WQ_k\|_F$ to update $B_k = (WQ_k)_+$ alternatively. 
It was suggested in \cite{HSS2014} to terminate the iteration when $|\langle B_{k-1},B_{k-1}-WQ_k\rangle|<\varepsilon$ with a given $\varepsilon$. This criterion may miss its global convergence. Since $f_k^{\rm APH} = \|B_k-WQ_k\|_F$ is monotone decreasing, we terminate the iteration if
\begin{align}\label{stop_APH}
    \frac{f_{k-1}^{\rm APH}-f_k^{\rm APH}}{f_k^{\rm APH}} < \epsilon_{r}^{\rm APH} \quad {\rm and }\quad
    \left\{\begin{array}{ll}
        f_k^{\rm APH}>\sqrt{\epsilon^{\rm APH}}, 
            & \mbox{ for local optimum, or}\\
        f_k^{\rm APH}<\epsilon^{\rm APH},
            & \mbox{ for global optimum}.
    \end{array}\right.
\end{align}
The left one can avoid meaningless iterations that do not provide acceptable decreasing on its objective function, while the right one can distinguish whether only a local minimum is achieved. The modification can increase the efficiency and is helpful for restarting this algorithm as did in REPM.

The AP-G aims to minimize $\min_{WP\geq 0,QQ^T = I} \|Q-P\|_F^2$ via alternative projection \cite{GD2018}: Set $P_{k+1} = Q_{k} - W^+(WQ_{k})_-$, an approximate solution to $\min_{WP\geq 0}\|Q_k-P\|_F^2$ given $Q_k$, and $Q_{k+1}=\arg\min_{QQ^T = I} \|Q-P_{k+1}\|_F^2$ given $P_{k+1}$. The algorithm terminates if $|\min(WQ_k)_{ij}|\leq \varepsilon$ with given $\varepsilon$, for example, $\varepsilon=10^{-15}$. We observe that the gap $\|Q_k-P_k\|_F$ matches the factorization error $\|A-(WQ_k)_+(WQ_k)_+^T\|_F$ better than $|\min(WQ_k)_{ij}|$ and that $\{\|Q_k-P_k\|_F\}$ is not monotone decreasing, A more efficient termination criterion is that 
\begin{align}\label{stop_APG}
    \epsilon_k = \min_{i\leq k}\|Q_i-P_i\|_F < \epsilon^{\rm APG}
\end{align}
with a given $\epsilon^{\rm APG}$. Let $i=i_k\leq k$ is the smallest index such that $\|Q_i-P_i\|_F = \epsilon_k$, {\it i.e.}, $\epsilon_{i_k}=\epsilon_k$. We terminate the iteration and take $Q_{i_k}$ as the output if (\ref{stop_APG}) is satisfied. If $\epsilon_{i_k}$ is not changed within $k^{\rm APG}$ iterations, we also terminate the algorithm to avoid unnecessary iterations.

Figure \ref{fig:APG-APH} illustrates the necessity of these modifications for AP-H and AP-G. In our experiments, we set $\epsilon_{r}^{\rm APH} = 10^{-7}$, $\epsilon^{\rm APH}=10^{-13}$, $\epsilon^{\rm APG} = 10^{-13}$, and $k^{\rm APG} = 5000$. We will also adopt the restart strategy for AP-H and AP-G, marked as RAP-H and RAP-G, within at most 10 number of restated implementation of the algorithm, each implementation is restricted at most 20 seconds.

\begin{figure}[t]
\centering
\includegraphics[width = \linewidth, height = 0.4\linewidth]{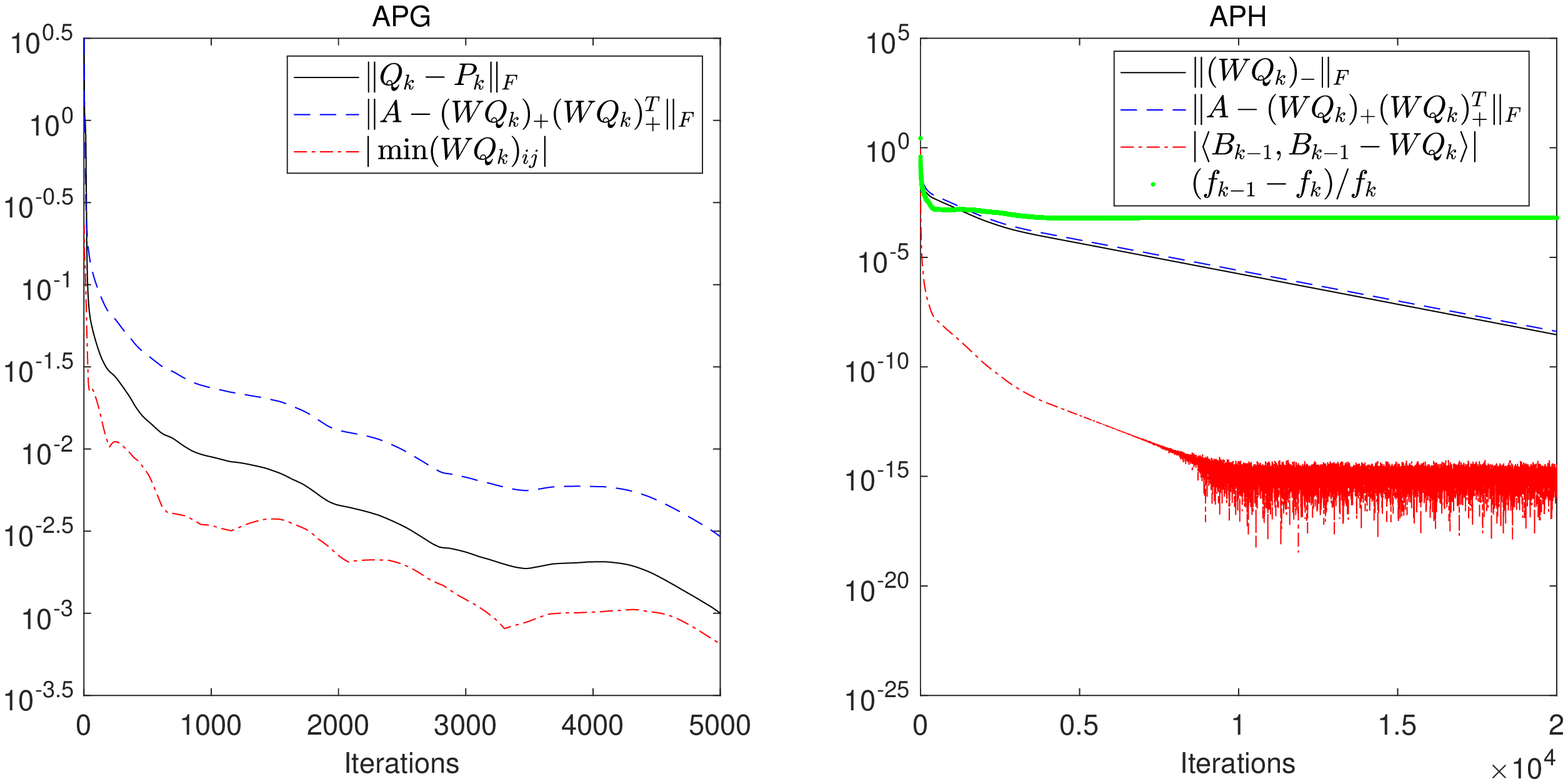}
\caption{Iterative behaviours of Algorithms AP-G and AP-H}
\label{fig:APG-APH}
\end{figure}

The CD method minimizes $\|A - BB^T\|_F^2$ via column-by-column optimization, together with component-wise optimization for column updating. ANLS minimizes 
$\|A - B_L B_R^T\|_F^2 + \alpha\|B_L-B_R\|_F^2$
with $B_L\geq 0$ and $B_R\geq 0$ alternatively, starting at a scaled random matrix $B_R$. It is more efficient than CD. However, both the two algorithms converge slowly in our experiments. Because of the weak point, the restart strategy is no longer suitable for CD or ANLS. 

\begin{figure}[t]
\centering
\includegraphics[width = \linewidth, height = 0.4\linewidth]{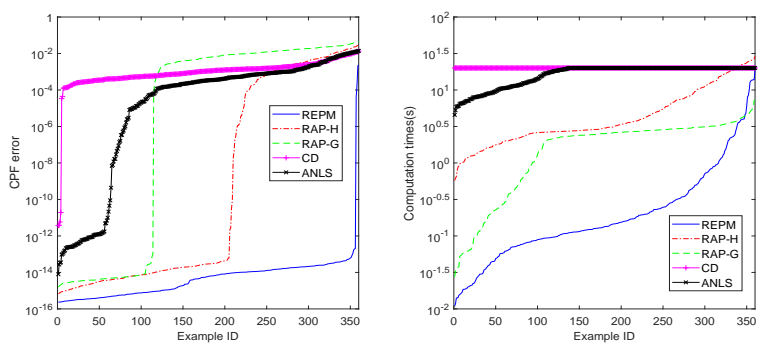}
\caption{Sorted CPF errors (left) and computational time (right) of the five algorithms on the 360 matrices}
\label{fig:comparison1}
\end{figure}

\subsection{Synthetic matrices with equal rank and cp-rank}
 
The tested matrices are randomly constructed with fixed $n = 200$, $r = 12$, and $b_{\min} = 0.1$, as did in the last section. For each type of the four column norm distributions, we choose 9 values for $s(B)$, varying from $1\%$ to $25\%$. 10 completely positive matrices are constructed for each type and each sparsity, and totally we have 360 testing matrices in this experiment. The CPF is not easy on some of these matrices, as shown in Tables \ref{tab:sparse}-\ref{tab:restart}.

The left panel of Figure \ref{fig:comparison1} plots the CPF errors of REPM, RAP-H, RAP-G, CD, and ANLS (without the restart strategy) on the 360 matrices. Each matrix is tested with random starting points.\footnote{As in REPM, we randomly construct an orthogonal matrix as the starting point for RAP-H and RAP-G since it performs better than the initial setting provided by the authors.} The REPM gives good results with relative CPF errors less than $10^{-13}$ on about 97.8\% matrices. The percentage is decreased to 56.9\% or 31.7\% for RAP-H or RAP-G, respectively. Meanwhile, the computational cost of REPM is much less than that of RAP-H and RAP-G. The percentages for CD and ANLS are very small. See the right panel of Figure \ref{fig:comparison1} for the comparisons on the computational time of these algorithms.

\begin{figure}[t]
\includegraphics[width = 1\linewidth]{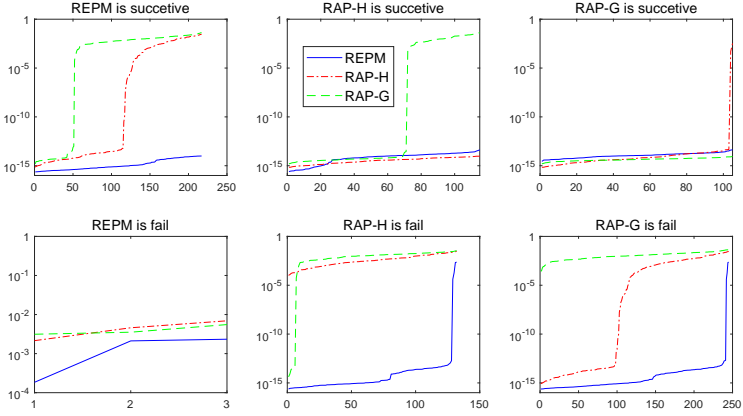}
\caption{CPF errors of the algorithms in six subsets of those tests corresponding to the CPF errors smaller than $10^{-14}$ (success set) or larger then $10^{-4}$ (fail set) of REPM, RAP-H, and RAP-G, respectively}
\label{fig:comparison2}
\end{figure}

The efficiency of an algorithm may depend on data sets. 
To make the comparison as fair as possible, for each of REPM, RAP-H, and RAP-G, we choose two sets of testing matrices that are especially suitable or unsuitable for the selected algorithm, in the sense that CPF errors are smaller than $10^{-14}$ (success set) or larger than $10^{-4}$ (fail set), respectively. Then, we compare the efficiency of other algorithms on matrices in the two special sets.
In the success sets of RAP-H or RAP-G, REPM is also successful. Meanwhile, in the success set of REPM, RAP-H and RAP-G fail on about 39.1\% and 76.0\% matrices, respectively. Conversely, in the sets of RAP-H and RAP-G fail, REPM has success rates 64.4\% and 67.4\% yet, respectively. The success rates for REPM can be increased to 97.0\% and 98.4\% if the accuracy is slightly decreased to $10^{-12}$. Figure \ref{fig:comparison2} shows the distributions of relative CPF errors for these three algorithms in each of the 6 sets.

As shown in Figure \ref{fig:comparison1}, CD fails to give a CPF with an acceptable accuracy -- the CPF errors are always larger than $10^{-4}$. The factorization error cannot be decreased when we restart CD. ANLS is slightly benefited from the restart strategy but the running time is significantly increased. For example, if we restart ANLS at most 10 times, each is restricted to run at most 20 seconds, the percentage of CPF with $\mbox{Error}(\tilde A)<10^{-10}$ is increased from 13.3\% to 17.5\% slightly, while the average computational time is unacceptably increased from 17.6 seconds to 171 seconds.

\subsection{Large-scale completely positive matrices}

In this comparison, we show how the efficiency of REPM, RAP-H, and RAP-G on matrices in large scale. We randomly construct 80 matrices in the scale $n = 20000$ and $r = 20$ with 4 different sparsity values $s(B) = 0$, 10\%, 20\%, and $30\%$ with each of the four distributions of column norms of $B$ as we did before.

The left panel of Figure \ref{fig:comparison3} plots the CPF errors of REPM, RAP-H and RAP-G. We terminate these algorithms when the limit on computational time $t_{\max} = 100$ seconds is touched. All the CPF errors of REPM are smaller than $10^{-12}$ on the 80 matrices. For RAP-H, there are only 37.5\% of CPF errors are smaller than $10^{-12}$. RAP-G performs poorly -- its CPF errors are larger than $10^{-8}$ on all the matrices. Meanwhile, the computational cost of REPM is much less than that of RAP-H and RAP-G. In the right of Figure \ref{fig:comparison3}, we also compare the computational time in seconds for these three algorithms on the 80 matrices. REPM is much faster than RAP-H and RAP-G.

\begin{figure}[t]
{\center
\includegraphics[width = 1\linewidth, height = 0.35\linewidth]{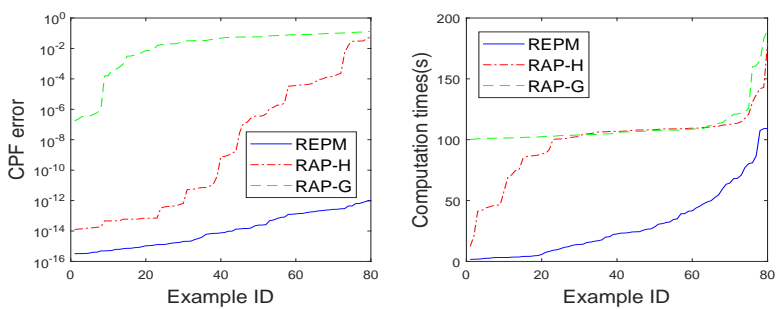}}
\caption{CPF errors (left) and computation time (right) of REPM, RAP-H, and RAP-G on 80 matrices}
\label{fig:comparison3}
\end{figure}

\subsection{Special matrices with cp-rank larger than rank}

It is not easy to construct a completely positive matrix with cp-rank larger than rank, except diagonally dominant symmetric nonnegative matrices \cite{K1987}. 
Here are the four special completely positive matrices given in the four papers \cite{BS2003,BB2003,GLL2017,BSU2015}, respectively, each has a cp-rank larger than its rank:
\begin{align*}
    A_1 = \left[\begin{matrix}
            6 & 3 & 3 & 0\\
            3 & 5 & 1 & 3\\
            3 & 1 & 5 & 3\\
            0 & 3 & 3 & 6\\
        \end{matrix}\right], \quad
    A_2 = WDW^T, \quad
    A_3 =  \left[\begin{matrix}
            I_k            & \frac{1}{k}J_k\\
            \frac{1}{k}J_k & I_k
        \end{matrix}\right], \quad    
    A_4 =  \left[\begin{matrix}
            G & H & H \\
            H & G & H \\
            H & H & G \\
        \end{matrix}\right],
\end{align*}
where both $I_k$ and $J_k$ are of order $k$, $I_k$ is identity, and $J_k$ has all entries equal to 1, 
\[
   W = \left[\begin{array}{rrr}
             1 &     0     & 1 \\
             a &  \sqrt{b} & 1 \\
            -b &  \sqrt{a} & 1 \\
            -b & -\sqrt{a} & 1 \\
             a & -\sqrt{b} & 1 \\
        \end{array}\right], \quad
   G =  \left[\begin{matrix}
              91 &  0  & 0  &  0\\
              0  &  42 & 0  &  0\\
              0  &  0  & 42 &  0\\
              0  &  0  & 0  &  42
        \end{matrix}\right], \quad
    H =  \left[\begin{matrix}
              19  &  24  & 24  &  24\\
              24  &  6   & 6   &  6\\
              24  &  6   & 6   &  6\\
              24  &  6   & 6   &  6
        \end{matrix}\right],
\]
$a = \frac{\sqrt{5}-1}{4}$, $b = \frac{\sqrt{5}+1}{4}$, and $D = \diag(2,\sqrt{5},\frac{3+\sqrt{5}}{2})$, The pairs $(r(A),r_{cp}(A))$ of these matrices are $(3,4)$, $(3,5)$, $(2k-1,k^2)$, and $(10,37)$, respectively. 
$A_2$ is a small example whose cp-rank can achieve the upper bound in the estimation $r_{cp}\leq \frac{r(r+1)}{2}-1$ and it has an explicit CPF \cite{BB2003}.  $A_3$ is an example whose cp-rank can be significantly larger than the matrix order \cite{GLL2017} since $r_{cp} = \frac{n^2}{4}$. It is diagonally dominant, and hence, has an explicit CPF by the factorization 
\begin{align}\label{factor dominant}
    A = \sum\limits_{1\leq i < j \leq n}a_{ij}(e_i+e_j)(e_i+e_j)^T + \diag(c_1,...,c_n),
\end{align}
given \cite{K1987} for any symmetric matrix $A=(a_{ij})$, 
where $c_i = a_{ii}-\sum_{j\neq i}a_{ij}$ and $e_i$ is the $i$-th column of the identity matrix of order $n$. 
Hence, if $A$ is nonnegative and diagonally dominant, {\it i.e.}, all $c_i\geq 0$, it must be completely positive since $A = BB^T$ with a nonnegative matrix $B$ of at most $r_+$ nonzero columns, where $r_+$ is the number of nonzero entries of $\{a_{ij}: i<j\}$ and $\{c_i\}$. For $A_3$, $r_+=k^2 = r_{cp}$.
$A_4$ is an example given in \cite{BSU2015} to show that there is a completely positive matrix whose cp-rank larger than $\frac{n^2}{4}$ without an CPF. In Appendix C, we give an explicit form of its CPF. Our exterior point method can give a strict CPF for each of these matrices almost exactly.

\begin{table}[t]
\centering
\resizebox{\textwidth}{!}{ 
\begin{tabular}{|c|c|cc|cc|cc|}\hline\hline
\multirow{2}{*}{Error($\tilde A$)} 
 & \multirow{2}{*}{\!\!Algorithm\!\!} 
 & \multicolumn{2}{c|}{$A_1$}  & \multicolumn{2}{c|}{$A_2$} 
 & \multicolumn{2}{c|}{$A_3\, (k=5)$}\\\cline{3-8}
 && T(s) & \# of iter.\!\! & T(s) & \# of iter.\!\! & T(s) & \# of iter.\!\!\\\hline
\multirow{3}{*}{\!\!$<\!10^{-14}$\!\!}
 & REPM & 0.002 &   128 & 0.004 &  283 & 0.104 &  6121\\
 & RAP-H & 0.091 & 16543 & 0.012 & 1493 & 1.510 & 89956\\
 & RAP-G & 0.07 & 11226 & 0.119 & 16007 & 2.512 & 140245\\\hline
\multirow{2}{*}{$>10^{-7}$}
 & CD   & 1.913  & 500001 & 6.077 & 1000001 & 20.000 & 353369\\
 & ANLS & 9.629 & 100000 & 9.878 & 100000  & 24.582 & 121204\\\hline
\hline
\end{tabular}}
\caption{Average values of the computational costs for the five algorithms on the special matrices}
\label{tab:special examples}
\end{table}

The REPM, RAP-H, and RAP-G perform very well on $A_1$, $A_2$, and $A_3$ with small $k$.\footnote{We use a randomly chosen row-orthonormal matrix as a restart point in REPM, rather than $-X$, and $\epsilon_1=10^{-7}$, keeping others unchanged.} Each of the three algorithms can quickly obtain a good CPF with an error smaller than $10^{-14}$. However, both the CD and ANLS fail to yield an acceptable CPF in the accuracy $10^{-7}$. 
We restrict the running time at most 20 seconds for CD and ANLS.\footnote{The REPM, RAP-H, and RAP-G do not touch the restriction on time.}
Because each iteration of the CD is faster than that of ANLS, we also use an additional limit $n_{\rm iter}$ for the iteration number in CD and ANLS: $n_{\rm iter}^{\rm CD} = \frac{1}{2}10^6$ for $A_1$ and $n_{\rm iter}^{\rm CD}= 10^6$ for the other three matrices, but $n_{\rm iter}^{\rm ANLS} = 10^5$. As mentioned before, the restart strategy is not suitable for CD since it seldom terminates before the two restrictions are not touched. In this experiment, we also restart the ANLS when the total iteration number is less than $n_{\rm iter}^{\rm ANLS}$.\footnote{The the total iteration number of restart ANLS could be slightly larger than $n_{\rm iter}^{\rm ANLS}$ but less $2n_{\rm iter}^{\rm ANLS}$.} 

In Table \ref{tab:special examples}, we list the average values and total iteration numbers, among 100 repeats on each matrix. The REPM converges much faster than RAP-H and RAPG. Besides the advantage point on computational time, the REPM is also more robust to the initial point. We also test the EPM, AP-H, and AP-G without the restart strategy on $A_1$, $A_2$ and $A_3$ with 100 tests for each algorithm. In these tests, EPM has a higher success rate, see Figure \ref{fig:comparison4}.

\begin{figure}[t]
\hspace{-20pt}\includegraphics[width = 1.1\linewidth, height = 0.3\linewidth]{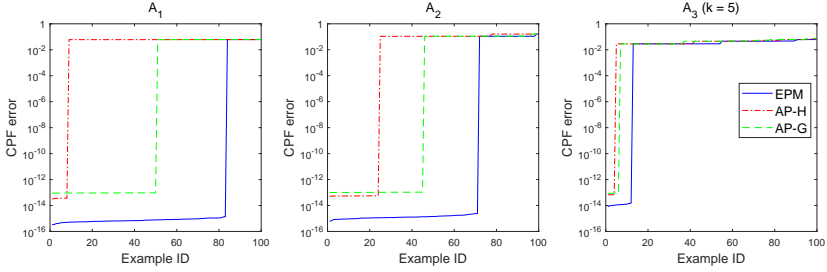}
\caption{Sorted CPF errors (left) of the three algorithms without the restart strategy on $A_1$, $A_2$ and $A_3$ ($k = 5$)}
\label{fig:comparison4}
\end{figure}

The cp-rank of $A_3$ is much larger than its rank or its matrix order if $k$ is slightly large, and so is it for $A_4$. The CPF is difficult in this case. 
We use a relative weak criterion for RAP-H and RAP-G: the algorithms are terminated when both the total running time $t_{\max}$ seconds and total iteration number $k_{\rm total}$ are touched, but a strict criterion for the REPM: it is terminated when the half of the time $t_{\max}$ is touched. We use different values of $t_{\max}$ and $k_{\rm total}$ for $A_3$ and $A_4$ as follows: For $A_3$ with $k=6,\, 8$, and 10, $t_{\max} = 40, 120$, and 240 seconds, $k_{\rm total} = 3\!\times\!10^{5},\, 5\!\times\!10^{5}$, and $8\!\times\!10^{5}$, respectively, and for $A_4$, $t_{\max} = 40$ seconds and $k_{\rm total} = 3\!\times\!10^{5}$.

Table \ref{tab:A4} shows the performance of REPM, RAP-H, and RAP-G on $A_3$ and $A_4$ on the average value of computational time and total iteration number among 100 times for each matrix. These three algorithms perform very well on $A_k$ with small $k=6$. As $k$ slightly increases to 8 or 10, the percentage of successful CPF of RAP-H or RAP-G decreases and the computational time increases quickly. RAP-H is a bit better than RAP-G on $A_4$, but its success rate is only 7\%. The REPM performs much better than RAP-H and RAP-G. On $A_3$ with $k=8$ and $A_4$, it always gives a CPF with error smaller than $10^{-12}$. On the difficult $A_3$ with $k=10$, it can also provide 47\% CPF within this accuracy.

\begin{table}[t] 
\centering
\resizebox{\textwidth}{!}{ 
\begin{tabular}{|c|c|c|rrr|rrr|rrr|}\hline\hline
\multirow{2}{*}{Error($\tilde A$)} & \multicolumn{2}{c|}{\multirow{2}{*}{Matrix}}
 & \multicolumn{3}{c|}{REPM} 
 & \multicolumn{3}{c|}{RAP-H} 
 & \multicolumn{3}{c|}{RAP-G}\\ \cline{4-12}
 &\multicolumn{2}{c|}{ }
 & Rate & Time & \# of iter
 & Rate & Time & \# of iter
 & Rate & Time & \# of iter.\\\hline
\multirow{4}{*}{$<10^{-12}$} & \multirow{2}{*}{$A_3$,} 
  &  6 & 100\% &  0.4s &  17764 &  100\% & 3.9s & 157958 & 100\% & 7.9 & 36718\\
 & \multirow{2}{*}{$k=$} 
 &  8 & 100\% & 6.7s & 209392 &  70\% & 53.9s & 1287981 & 51\% & 54.7 & 148167\\
 && 10 &   47\% & 54.5s & 772537 &  3\% & 104.5s & 1476917 & 5\% & 147.7 & 2315969\\ \cline{2-3}
 & \multicolumn{2}{c|}{$A_4$} & 100\% &  1.8s &  100032 &  7\% & 15.35s & 672146 & 0\% & -- & --\\
 \hline
\multirow{4}{*}{$>10^{-4}$} & \multirow{2}{*}{$A_3$,}
  &  6 &  0\% &  --   &    --  &  0\% &  -- &  -- & 0\% & -- & --\\
 & \multirow{2}{*}{$k=$}
  &  8 & 0\% & -- & -- & 30\% & 120.1s & 2869192 & 49\% & 120.3 &3252231 \\
 && 10 & 53\% & 120.1s & 1668324 & 97\% & 240.3s & 3365013 & 95\%& 240.7& 3732500\\ \cline{2-3}
 & \multicolumn{2}{c|}{$A_4$} & 0\% &  -- &  -- &  93\% & 40.1s & 1740837 & 100\% & 40.1s & 1590793\\
 \hline\hline
\end{tabular}
}
\caption{Performance of REPM, RAP-H, and RAP-G on the matrix $A_3$ with $k=6,8,10$ and $A_4$}
\label{tab:A4}
\end{table}

\section{Conclusions}

In this paper, we tried the idea of exterior point method for addressing the CPF problem numerically. The proposed optimization model and its iterative solver via a modified NCG can implement the exterior point method. In the numerical experiments reported in this paper, the exterior point method performs much better than the algorithms in the literature. We discussed some potential issues that may affect the CPF via a lot of numerical experiments by our algorithm. Some phenomena are interesting and might be helpful for further analysis on this topic.
However, we just touched a small angle of the ice mountain. For instance, 
The special rank-sparsity boundary mentioned in Phenomenon B may determine how hard the CPF is, but we have no idea to verify its existence or characterize such a boundary theoretically.
Besides this, the approximate cp-rank deficiency may also result in a difficult CPF. For our algorithm or its restart version on difficult CPFs, it deserves to explore an efficient initial setting. The weak CPF problem is relatively easier than the (strict) CPF. It may also be an interesting topic to transform a weak CPF to a strict CPF efficiently. It is worth further working on these topics.

\begin{acknowledgements}
The work was supported in part by NSFC project 11971430 and Major Scientific Research Project of Zhejiang Lab (No. 2019KB0AB01).
\end{acknowledgements}

\appendix
\section{Transversal intersection of $\mathbb Q$ and $\mathbb{P}$}

\begin{proposition}\label{prop:QP}
The Stiefel manifold $\mathbb{Q}_r = \{Q \in\mathbb{R}^{r\times r_+}: QQ^T = I_r\}$ transversally intersects the submanifold $\mathbb{P}_r = \{P\in \mathbb{R}^{r\times r_+}: WP\geq 0\}$ at $Q\in \mathbb{Q}_r\cap \mathbb{P}_r$ if and only if for any nonnegative matrix $Y\in {\mathbb{R}}^{n\times r}$, $W^TYQ^T$ is not symmetric or its trace is not zero when it is not zero. 
\end{proposition}
\begin{proof}
It is known that the normal spaces of $\mathbb{Q}_r$ at $Q$ is ${\cal N}_{\mathbb{Q}_r}(Q) 
    = \big\{SQ:\ S\in \mathbb{R}^{r\times r}, S=S^T\big\}$. 
If we also have $Q\in\mathbb P_r$, the normal cone of $\mathbb{P}_r$ at $Q$ is defined as
$
    {\cal N}_{\mathbb{P}_r}(Q) 
    = \big\{N:\ \langle N, P-Q\rangle\leq 0,\ P\in \mathbb{P}_r\big\}.
$ 
Since $W$ is of full column rank, we can rewrite any $N\in{\mathbb R}^{r\times r_+}$ as $N  = -W^TY$ with a $Y\in{\mathbb R}^{n\times r_+}$.\footnote{$Y$ is not unique in the representation $W^TY$.} The restriction for ${\cal N}_{\mathbb{P}_r}(Q)$ becomes $\langle Y, WP-WQ\rangle\geq 0$ for $P\in \mathbb{P}$.
Since $Q\in\mathbb{P}$, choosing $P = 2Q$ and $P = 2^{-1}Q$ in the restriction, it is equivalent to $\langle Y, WQ\rangle = 0$ and $\langle Y, WP\rangle \geq 0$ for $P\in \mathbb{P}$. Furthermore, we can restrict $Y$ to be nonnegative, which makes the inequality $\langle Y, WP\rangle \geq 0$ hold automatically for all $P\in\mathbb{P}$, and hence, we can represent 
\[
    {\cal N}_{\mathbb{P}_r}(Q) 
 	= \big\{-W^TY:\ Y\in \mathbb{R}^{m\times r_+},\ Y\geq 0,\ \langle Y,WQ\rangle = 0\big\}.
\]

To show the existence for a fixed $W^TY$, we assume $W^TY\neq W^T\tilde Y$ for any $\tilde Y\geq 0$. Hence, $W^TY$ and the closed cone $\{W^T\tilde Y: \tilde Y\geq 0\}$ are separated. That is, there is an $H\in{\mathbb R}^{r\times r_+}$ such that $\langle Y, WH\rangle < \langle \tilde Y, WH\rangle$ for all $\tilde Y\geq 0$.
Let $\tilde Y_t\geq 0$ with zero entries except $(\tilde Y_t)_{ij}=t>0$ for arbitrary index pair $(i,j)$. We get that as $t\to 0$, $(WH)_{ij} = \frac{1}{t}\langle \tilde Y_t, WH\rangle>\frac{1}{t}\langle Y, WH\rangle\to 0$, which implies $(WH)_{ij}\geq 0$. That is, $H\in\mathbb{P}$. A contradiction follows immediately as that $0\leq \langle Y, WH\rangle < \langle \tilde Y, WH\rangle=0$ if we set $\tilde Y=0$. 

Hence, $\mathbb{Q}_r$ intersects $\mathbb{P}_r$ transversally at $Q\in \mathbb{Q}_r\cap \mathbb{P}_r$, that is by definition,
${\cal N}_{\mathbb{Q}_r}(Q)$ and $-{\cal N}_{\mathbb{P}_r}(Q)$ are intersected at the origin only, is equivalent to that if $SQ=W^TY$ for a symmetric $S\in \mathbb{R}^{r\times r}$ and a nonnegative $Y\in \mathbb{R}^{n\times r_+}$ satisfying $\langle WQ,Y\rangle = 0$, we must have $S = W^TYQ^T=0$. It is also equivalent to that for any $Y\geq 0$ of order $r\times r_+$, there is not a symmetric $S$ equal to $W^TYQ^T$, {\it i.e.}, $W^TYQ^T$ is not symmetric, or $Y\notin -{\cal N}_{\mathbb{P}_r}(Q)$, {\it i.e.}, the trace of $W^TYQ^T$ is not zero when it is not zero.
$\hfill\square$
\end{proof}

\section{Proof of Theorem \ref{thm:second order}}

\begin{proof}
For a stationary point $X$ of $f$, let $T_0$ and $T_-$ be the indicator matrices of the zero entries and negative entries of $WX$, respectively. Consider a sufficiently small neighborhood $N(X)$ of $X$, in which $(W\tilde X)_{ij}<0$ if $(WX)_{ij}<0$, and $(W\tilde X)_{ij}>0$ if $(WX)_{ij}>0$. For $\tilde X = X+\Delta\in N(X )$, since $(WX )\odot T_0=0$, we get that 
\[
    (W\tilde X)_- = (W\tilde X)\odot T_-+(W\Delta)_-\odot T_0, 
\]
and $\|(W\tilde X)_-\|_F^2 = \|(W\tilde X)\odot T_-\|_F^2+ \|(W\Delta)_-\odot T_0\|_F^2$. Let 
\[
    f_-(\tilde X) =  \frac{1}{4}\|\tilde X\tilde X^T-I\|_F^2 +\frac{\lambda}{2}\|(W\tilde X)\odot T_-\|_F^2.
\]
It gives $\nabla f_-(\tilde X) = (\tilde X\tilde X^T-I)\tilde X+\lambda W^T\big((W\tilde X)\odot T_-\big)$. Hence,
\begin{align}\label{ff_T}
    f(\tilde X) &= f_-(\tilde X) + \|(W\Delta)_-\odot T_0\|_F^2,\quad
    \nabla f(\tilde X)  = \nabla f_-(\tilde X) +\lambda W^T\big((W\Delta)_-\odot T_0\big).
\end{align}
Obviously, $f_-(X )=f(X )$ and $\nabla f_-(X )=\nabla f(X )=0$. 
It is easy to verify that
\begin{align*}
    f_-(\tilde X) - f_-(X ) 
    & = \frac{\lambda}{2}\|(W\Delta)\odot T_-\|_F^2
    + \frac{1}{2}\big\langle X X ^T-I, \Delta\Delta^T\big\rangle
    + \frac{1}{4}\|\Delta X ^T+X \Delta^T+\Delta\Delta^T\|_F^2\\
    & = \langle{\cal A}(\Delta),\Delta\rangle
      + \frac{1}{2}\big\langle \Delta X ^T+X \Delta^T,\Delta\Delta^T\big\rangle
      + \frac{1}{4}\|\Delta\Delta^T\|_F^2,
\end{align*}
where $\langle{\cal A}(\Delta),\Delta\rangle$ is a quadratic form with the linear mapping ${\cal A}: {\mathbb R}^{r\times r_+}\to {\mathbb R}^{r\times r_+}$, 
\[
    {\cal A}(\Delta) = (X X ^T-I)\Delta + \Delta X ^TX +X \Delta^TX 
    +\lambda W^T\big((W\Delta)\odot T_-\big).
\]
It is not difficult to rewrite $\langle{\cal A}(\Delta),\Delta\rangle$ as a quadratic form $\delta^T\!S\delta$ with a symmetric matrix $S$ and the the vector representation $\delta = v(\Delta)$ of the matrix $\Delta$. If $S$ is positive definite, the quadratic form must be also positive definite, and hence, $X $ is a local minimizer of $f_-$. If $S$ is positive semidifinite, let $\{\delta^{(k)} = v(\Delta^{(k)})\}$ be the $rr_+$ unit eigenvectors of $S$ corresponding to eigenvalues $\{\lambda_k\}$ in ascending order and the first $m$ ones are zeros. By the assumption, $\Delta^{(k)} X ^T+X (\Delta^{(k)})^T = 0$ for $k\leq m$. Then, for $\tilde X = X +\Delta$ with a nonzero $\Delta = \sum_k\alpha_k\Delta^{(k)}$, 
\begin{align*}
    f_-(\tilde X) - f_-(X ) 
    &=\langle{\cal A}(\Delta),\Delta\rangle
    + \frac{1}{2}\big\langle \Delta X ^T+X \Delta^T,\Delta\Delta^T\big\rangle
    + \frac{1}{4}\|\Delta\Delta^T\|_F^2\\
    &= \sum_{k>m}\Big\{\alpha_k^2\lambda_k
    + \alpha_k\big\langle \Delta^{(k)} X ^T,\Delta\Delta^T\big\rangle\Big\}
    + \frac{1}{4}\|\Delta\Delta^T\|_F^2.
\end{align*}
Since 
$\big|\sum_{k>m}\alpha_k\big\langle \Delta^{(k)} X ^T,\Delta\Delta^T\big\rangle\big|
= \big|\big\langle\sum_{k>m}\alpha_k\Delta^{(k)},\Delta\Delta^TX \big\rangle\big|
\leq \|\sum_{k>m}\alpha_k\Delta^{(k)}\|_F\|\Delta\Delta^T\|_F\|X \|_2$ and by the orthogonality of $\{\delta^{(k)}\}$,
$\|\sum_{k>m}\alpha_k\Delta^{(k)}\|_F=\sqrt{\sum_{k>m}\alpha_k^2}$, we get
\[
    \big|\sum_{k>m}\alpha_k\big\langle \Delta^{(k)} X ^T,\Delta\Delta^T\big\rangle\big|
    \leq \sqrt{\sum_{k>m}\alpha_k^2}\|X \|_2\|\Delta\Delta^T\|_F
    \leq \sum_{k>m}\alpha_k^2\|X \|_2^2+ \frac{1}{4}\|\Delta\Delta^T\|_F^2.
\]
Therefore, $f_-(\tilde X) - f_-(X ) \geq \sum_{k>m}\alpha_k^2(\lambda_k-\|X \|_2^2)\geq 0$. Hence, $f(\tilde X)\geq f_-(\tilde X)\geq f_-(X ) = f(X )$. That is, $X $ must be also a local minimizer of $f$.
$\hfill\square$
\end{proof}

\section{The strict CPF of $A_4$}

Let $ e_i$ be the $i$-th column of the identity matrix $I_4$ of order 4, and let
\[
    B_1 = \left[\begin{array}{c} e_4\\ e_4\\ e_4\end{array}\right],\ 
    B_2 = \sqrt{5}\left[\begin{array}{ccc} e_4 & e_1 & e_4\\ e_1 & e_4 & e_4\\ e_4 & e_4 & e_1\end{array}\right],\ 
    B_3 = \sqrt{7}\left[\begin{array}{ccc} e_1 & e_4 & e_1\\ e_4 & e_1 & e_1\\ e_1 & e_1 & e_4\end{array}\right],\ 
    B_4 = \sqrt{6}\left[\begin{array}{c} B_{14}\\ B_{24}\\ B_{34}\end{array}\right],
\]
where $B_{14} = [e_1{\bf 1}_{12}^T, e_2{\bf 1}_7^T, e_3{\bf 1}_7^T, e_4{\bf 1}_4^T]$, where ${\bf 1}_k$ is a $k$-dimensional column vector of all ones, $B_{24} = B_{14}P_2^T$, and $B_{34} = B_{14}P_3^T$ with two permutation matrices $P_2$ and $P_3$ such that 
\[
    B_{i4}B_{i4}^T = \diag(12,7,7,4),\ i=1,2,3,\quad 
    B_{i4}B_{j4}^T = \left[\begin{array}{cccc}
     2 & 4 & 4 & 2\\
     4 & 1 & 1 & 1\\
     4 & 1 & 1 & 1\\
     2 & 1 & 1 & 0\\ \end{array}\right],\ i\neq j.
\]
One can verify that $A_4 = BB^T$ with $B=[B_1,B_2,B_3, B_4]$, a strict CPF of $A_4$.

\end{document}